\documentclass[final]{siamart}

\usepackage{epsfig}
\usepackage{latexsym}
\usepackage{amsmath}
\usepackage{amsfonts}
\usepackage{amssymb}
\usepackage{stmaryrd}
\usepackage{dsfont}
\usepackage{curves}
\usepackage{color}
\usepackage{version}

\newcommand{\tr}{{\rm tr}}
\newcommand{\as}{{\rm as}}

\newcommand{\bdev}{{\bf dev}}

\renewcommand{\div}{{\rm div}}

\newcommand{\R}{\mathds{R}}

\newcommand{\bpsi}{\mbox{\boldmath$\psi$}}

\newcommand{\bXi}{\mbox{\boldmath$\Xi$}}
\newcommand{\bepsilon}{\mbox{\boldmath$\varepsilon$}}
\newcommand{\btau}{\mbox{\boldmath$\tau$}}
\newcommand{\bchi}{\mbox{\boldmath$\chi$}}
\newcommand{\bxi}{\mbox{\boldmath$\xi$}}

\newcommand{\brho}{\mbox{\boldmath$\rho$}}
\newcommand{\bsigma}{\mbox{\boldmath$\sigma$}}

\newcommand{\bSigma}{\mbox{\boldmath$\Sigma$}}

\newcommand{\bnabla}{\mbox{\boldmath$\nabla$}}
\newcommand{\bzero}{{\bf 0}}
\newcommand{\bff}{{\bf f}}
\newcommand{\bv}{{\bf v}}

\newcommand{\bz}{{\bf z}}
\newcommand{\bb}{{\bf b}}
\newcommand{\bu}{{\bf u}}

\newcommand{\bg}{{\bf g}}
\newcommand{\bq}{{\bf q}}

\newcommand{\bB}{{\bf B}}

\newcommand{\bI}{{\bf I}}
\newcommand{\bJ}{{\bf J}}

\newcommand{\bV}{{\bf V}}

\newcommand{\bn}{{\bf n}}
\newcommand{\bt}{{\bf t}}
\newcommand{\cA}{{\cal A}}

\newcommand{\cE}{{\cal E}}

\newcommand{\cP}{{\cal P}}

\newcommand{\cT}{{\cal T}}
\newcommand{\cV}{{\cal V}}

\newcommand{\colb}{\color{blue}}

\begin{document}

\title{A posteriori error estimation for planar linear elasticity by stress reconstruction}

\author{Fleurianne Bertrand%
\thanks{Fakult\"at f\"ur Mathematik, Universit\"at Duisburg-Essen, Thea-Leymann-Str. 9,
        45127 Essen, Germany
        (\email{fleurianne.bertrand@uni-due.de}, \email{marcel.moldenhauer@uni-due.de}, \email{gerhard.starke@uni-due.de}).
        The authors gratefully acknowledge support by the Deutsche Forschungsgemeinschaft in the Priority Program SPP 1748
        `Reliable simulation techniques in solid mechanics. Development of nonstandard discretization methods, mechanical and
        mathematical analysis' under the project STA 402/12-1.        
        }
        \and Marcel Moldenhauer\footnotemark[1]
        \and Gerhard Starke\footnotemark[1]
        }
        
\date{}
\maketitle

\begin{abstract}
  The nonconforming triangular piecewise quadratic finite element space by Fortin and Soulie can be used for the displacement
  approximation and its combination with discontinuous piecewise linear pressure elements is known to constitute a stable
  combination for incompressible linear elasticity computations. In this contribution, we extend the stress reconstruction
  procedure and resulting guaranteed a posteriori error estimator developed by Ainsworth, Allendes, Barrenechea and Rankin
  \cite{AinAllBarRan:12} and by Kim \cite{Kim:12a} to linear elasticity. In order to get a guaranteed reliability bound with respect
  to the energy norm involving only known constants, two modifications are carried out: (i) the stress reconstruction in
  next-to-lowest order Raviart-Thomas spaces is modified in such way that its anti-symmetric part vanishes in average on each
  element; (ii) the auxiliary conforming approximation is constructed under the constraint that its divergence coincides with the
  one for the nonconforming approximation. An important aspect of our construction is that all results hold uniformly in the
  incompressible limit. {\colb Global} efficiency is also shown and the effectiveness is illustrated by adaptive computations
  involving different Lam\'e parameters including the incompressible limit case.
\end{abstract}



\headers{A posteriori error estimation by stress reconstruction}
              {F. Bertrand, M. Moldenhauer, and G. Starke}

\section{Introduction}

\label{sec-introduction}

This paper is concerned with the nonconforming triangular finite element space of piecewise quadratic functions applied to linear
elasticity in two space dimensions. This finite element space possesses some peculiarities due to the existence of a non-zero
quadratic polynomial which vanishes at both Gauss points on all three boundary edges. A suitable basis for this space was
constructed by Fortin and Soulie in \cite{ForSou:83} and the corresponding generalization to three dimensions by Fortin in
\cite{For:85}. The special structure of this nonconforming space leads to some advantageous properties of the piecewise
gradients which have been exploited for the purpose of flux or stress reconstruction by Ainsworth and co-workers in
\cite{AinRan:08a} and \cite{AinAllBarRan:12} and by Kim in \cite{Kim:12a}. Roughly speaking, their reconstruction algorithm
based on the quadratic nonconforming finite element space remains more local and requires less computational work than
similar approaches for more general finite element spaces.

Our contribution with this work consists in the modification of this approach to the stress reconstruction associated with
incompressible linear elasticity and the corresponding guaranteed a posteriori error estimation. The applicability of the
procedure in \cite{AinAllBarRan:12} and \cite{Kim:12a} is limited to bilinear forms involving the full gradient including the
Stokes system which is equivalent to incompressible linear elasticity if Dirichlet conditions are prescribed on the entire boundary.
For incompressible linear elasticity with traction forces prescribed on some part of the boundary, the symmetric gradient needs
to be used instead and this leads to complications associated with the anti-symmetric part of the stress reconstruction. In order
to keep the constants associated with the anti-symmetric stress part under control, a modification like the one presented in this
work needs to be done. Of course, one could perform the stress reconstruction in a symmetric $H (\div)$-conforming stress
space as it is done in \cite{NicWitWoh:08} and \cite{AinRan:10} based on the Arnold-Winther elements \cite{ArnWin:02}.
But this complicates the
stress reconstruction procedure significantly compared to the Raviart-Thomas elements of next-to-lowest order used here.
This is particularly true in three dimensions where the symmetric $H (\div)$-conforming finite element space from
\cite{ArnAwaWin:08} involves polynomials of degree 4 with 162 degrees of freedom per tetrahedron.

Although we restrict our investigation to two space dimensions in this paper, the treatment of three-dimensional elasticity
problems is our ultimate goal. The ingredients of our approach can be generalized to the three-dimensional case on
tetrahedral elements, in some aspects in a straightforward way, in other aspects with complications. We are
convinced that this is one of the most promising routes towards an effective guaranteed a posteriori error estimator for
three-dimensional incompressible linear elasticity. We will therefore remark on generalizations to three dimensions at the end
of our paper and also point out where the generalization is not so straightforward.

With respect to a triangulation $\cT_h$ with the corresponding set of edges denoted by $\cE_h$, the nonconforming finite
element space of degree 2 is defined by
\begin{equation}
  \begin{split}
    V_h & = \{ v_h \in L^2 (\Omega) : \left. v_h \right|_T \in P_2 (T) \mbox{ for all } T \in \cT_h \: , \\
    & \hspace{2cm}
    \langle \llbracket v_h \rrbracket_E , z \rangle_{L^2 (E)} = 0 \mbox{ for all } z \in P_1 (E) \: , \: E \in \cE_h \cap \Omega \: , \\
    & \hspace{2cm}
    \langle v_h , z \rangle_{L^2 (E)} = 0 \mbox{ for all } z \in P_1 (E) \: , \: E \in \cE_h \cap \Gamma_D \} \: ,
  \end{split}
  \label{eq:nonconforming_definition}
\end{equation}
where $\llbracket \: \cdot \: \rrbracket_E$ denotes the jump across the side $E$. In contrast to the lowest-order case (the
nonconforming Crouzeix-Raviart elements), the implementation is not straightforward due to the existence of non-trivial
finite element functions whose support is restricted to one element. On the other hand, the quadratic nonconforming finite
element space has the remarkable property that the continuity equation is satisfied in average on each element and that
the jump of the associated directional derivative in normal direction is zero in average across sides. These properties and
the construction of a suitable basis were already contained in the landmark paper by Fortin and Soulie \cite{ForSou:83}
(and generalized to the three-dimensional case in Fortin \cite{For:85}).

One of the strengths of the quadratic nonconforming finite element space is that it provides a stable mixed method for the
approximation of Stokes or incompressible linear elasticity if it is combined with the space of discontinuous piecewise linear
functions for the pressure. Due to the fact that it satisfies a discrete Korn's inequality also in the presence of traction boundary
conditions, it is among the popular mixed finite element approaches in common use (cf.
\cite[Sects. 8.6.2 and 8.7.2]{BofBreFor:13}). It received increased attention in recent years in the context of a posteriori error
estimation based on reconstructed fluxes or stresses (cf. \cite{AinRan:08a}, \cite{AinAllBarRan:12} and \cite{Kim:12a}). The
special properties associated with average element-wise and side-wise conservation of mass or momentum mentioned above
also lead to simplified flux or stress reconstruction algorithms.
As will be explained in detail in Section \ref{sec-first_estimator}, the direct application of the approach from
\cite{AinAllBarRan:12} and \cite{Kim:12a} to the displacement-pressure formulation of incompressible linear elasticity leaves two
global constants in the reliability bound which remain dependent on the geometry of the domain. In order to overcome this
dependence two modifications of the error estimator are performed. Firstly, we construct a correction to the stress reconstruction
such that the element-wise average of its anti-symmetric part vanishes. This enables us to multiply the anti-symmetry term in
the estimator by a constant associated with an element-wise Korn's inequality which can be computed from the element shape.
Such a shape-dependent Korn's inequality was used by Kim \cite{Kim:11b} in association with a posteriori error estimation for
a stress-based mixed finite element approach. Secondly, the auxiliary conforming approximation is
computed under the additional constraint that its divergence coincides with the nonconforming one. Both steps can be carried
out in a local fashion using a vertex-patch based partition of unity. This procedure ensures that the properties of the resulting
guaranteed a posteriori error estimator remain uniformly valid in the incompressible limit.

A posteriori error estimation based on stress reconstruction has a long history with ideas dating back at least as far as
\cite{LadLeg:83} and
\cite{PraSyn:47}. Recently, a unified framework for a posteriori error estimation based on stress reconstruction for the Stokes
system was carried out in \cite{HanSteVoh:12} (see also \cite{ErnVoh:15} for polynomial-degree robust estimates). These
two references include the treatment of nonconforming methods and both of them contain a historical perspective with a long list
of relevant references.



The outline of this paper is as follows. In the next section we review the displacement-pressure formulation for planar linear
elasticity, its approximation using quadratic nonconforming finite elements and the associated stress reconstruction
procedure in Raviart-Thomas spaces. Based on this, a preliminary version of our a posteriori error estimator is derived in
Section \ref{sec-first_estimator}. Section \ref{sec-guaranteed_estimator} provides an improved and guaranteed a posteriori
estimator where all the constants in the reliability bound are known and depend only on the shape-regularity of the
triangulation. To this end, reconstructed stresses with element-wise average symmetry are required. A procedure for this
construction is described in Section \ref{sec-average_symmetry}. Section \ref{sec-divergence_conforming} presents an
approach to the construction of a conforming approximation with divergence constraint which is also needed for the error
estimator of Section \ref{sec-guaranteed_estimator}. {\colb Global} efficiency is established in Section \ref{sec-efficiency}. Finally,
Section \ref{sec-computational} presents the computational results for the well-known Cook's membrane problem for
different Lam\'e parameters including the incompressible limit.

\section{Displacement-pressure formulation for incompressible linear elasticity and weakly symmetric stress reconstruction}

\label{sec-linear_elasticity_stress}

On a bounded domain $\Omega \subset \R^2$, assumed to be polygonally bounded such that the union of elements in the
triangulation $\cT_h$ coincides with $\Omega$, the boundary is split into $\Gamma_D$ (of positive length) and
$\Gamma_N = \partial \Omega \backslash \Gamma_D$. The boundary value problem of (possibly) incompressible linear
elasticity
consists in the saddle-point problem of finding $\bu \in H_{\Gamma_D}^1 (\Omega)^2$ and $p \in L^2 (\Omega)$ such that
\begin{equation}
  \begin{split}
    2 \mu \: ( \bepsilon (\bu) , \bepsilon (\bv) ) + ( p , \div \: \bv ) 
    & = ( \bff , \bv ) + \langle \bg , \bv \rangle_{L^2 (\Gamma_N)} \: ,\\
    ( \div \: \bu , q ) - \frac{1}{\lambda} ( p , q ) & = 0
  \end{split}
  \label{eq:disp_pressure}
\end{equation}
holds for all $\bv \in H_{\Gamma_D}^1 (\Omega)^2$ and $q \in L^2 (\Omega)$. $\bff \in L^2 (\Omega)^2$ and
$\bg \in L^2 (\Gamma_N)^2$ are prescribed volume and surface traction forces, respectively.
$\mu$ and $\lambda$ are the Lam\'e parameters characterizing the material properties, where
{\colb we may assume $\mu$ to be on the order of one and}
our particular interest lies in
large values of $\lambda$ associated with near-incompressibility. In the limit $\lambda \rightarrow \infty$,
(\ref{eq:disp_pressure}) turns to the Stokes system modelling incompressible fluid flow. For $\Gamma_D = \partial \Omega$,
the Stokes system may then be restated with the full gradients $\bnabla \bu$ and $\bnabla \bv$ and the approaches from
\cite{AinAllBarRan:12}, \cite{Kim:12a} may be applied directly component-wise. In general, however, the symmetric part of the
gradient appearing in (\ref{eq:disp_pressure}) cannot be avoided which complicates the derivation of an a posteriori error
estimator as we shall see below.

The resulting finite-dimensional saddle-point
problem is then to find $\bu_h \in \bV_h$ and $p_h \in Q_h$ such that
\begin{equation}
  \begin{split}
    2 \mu ( \bepsilon (\bu_h) , \bepsilon (\bv_h) )_h + ( p_h , \div \: \bv_h )_h &
    = ( \cP_h \bff , \bv_h ) + \langle \cP_h^{\Gamma_N} \bg , \bv_h \rangle_{\Gamma_N} \\
    ( \div \: \bu_h , q_h )_h - \frac{1}{\lambda} ( p_h , q_h ) & = 0
  \end{split}
  \label{eq:disp_pressure_nonconforming}
\end{equation}
holds for all $\bv_h \in \bV_h$ and $q_h \in Q_h$. The $L^2 (\Omega)$-orthogonal projection $\cP_h$ is meant
component-wise and $\cP_h^{\Gamma_N}$ stands for the (component-wise) $L^2 (\Gamma_N)$-orthogonal projection onto
piecewise linear functions without continuity restrictions. The notation $( \: \cdot \: , \: \cdot \: )_h$ stands for the piecewise
$L^2 (T)$ inner product, summed over all elements $T \in \cT_h$.
The finite element combination of using nonconforming piecewise
quadratic elements for (each component of) the displacements in $\bV_h$ with piecewise linear discontinuous elements
for the pressure in $Q_h$ is particularly attractive. Similarly to the Taylor-Hood pair of using conforming quadratic
finite elements for $\bV_h$ combined with continuous linear finite elements for $Q_h$, it achieves the same optimal
approximation order for both variables. This leads to the same quality of the stress approximation
$\bsigma_h (\bu_h,p_h) = 2 \mu \bepsilon (\bu_h) + p_h \: \bI$ with respect to the $L^2 (\Omega)$ norm. At the expense of an
increased number of degrees of freedom compared to the Taylor-Hood elements (about 1.6 times in 2D, almost 4 times in 3D)
on the same triangulation, the nonconforming approach offers advantages for the stress reconstruction which justifies its use.

For the actual computation of the nonconforming finite element approximation $u_h$, a basis of the space $\bV_h$ is required.
Such a basis, consisting of functions with local support, was derived in \cite{ForSou:83} and \cite{For:85} for the two- and
three-dimensional situation, respectively. The construction uses the fact that $\bV_h = \bV_h^C + \bB_h^{NC}$ with the
conforming subspace $\bV_h^C$ of continuous piecewise quadratic functions and a nonconforming bubble space $\bB_h^{NC}$.
In the two-dimensional case, the nonconforming bubble space is given by
\begin{equation}
  \begin{split}
    \bB_h^{NC} & = \{ \bb_h^{NC} \in L^2 (\Omega)^2 : \left. \bb_h^{NC} \right|_T \in P_2 (T)^2 \mbox{ for all } T \in \cT_h \: , \\
                        & \hspace{2cm} \langle \bb_h^{NC} , \bz \rangle_{L^2 (E)} = 0 \mbox{ for all } \bz \in P_1 (E)^2 \: , \: E \in \cE_h \}
                        \: ,
  \end{split}
\end{equation}
i.e., there is exactly one nonconforming bubble function in $\bB_h^{NC}$ for each displacement component per triangle. We
denote the corresponding space restricted to an element $T \in \cT_h$ by $\bB_h^{NC} (T)$. It should be
kept in mind that the representation $\bV_h = \bV_h^{C} + \bB_h^{NC}$ is not a direct sum, in general. For example, globally
constant functions can be expressed in two different ways in these subspaces, in general. In any case, if $\Gamma_D$ is a
connected subset of positive length of $\partial \Omega$, a basis of $\bV_h$ consisting of nonconforming
bubble functions in $\bB_h^{NC}$ and conforming nodal basis functions in $\bV_h^C$ can be selected.

The well-posedness of the discrete linear elasticity system (\ref{eq:disp_pressure_nonconforming}) relies on the discrete
Korn's inequality
\begin{equation}
  \| \bnabla \bv \|_h \leq C_K \| \bepsilon (\bv) \|_h \mbox{ for all } \bv \in H_{\Gamma_D}^1 (\Omega)^2 + \bV_h \: ,
  \label{eq:discrete_Korn}
\end{equation}
which is satisfied with some constant $C_K$
by the nonconforming quadratic elements (under our assumption that $\Gamma_D$ is a subset of the
boundary with positive length) due to \cite[Thm. 3.1]{Bre:03b}. It is well-known that the linear nonconforming elements by
Crouzeix-Raviart do not satisfy the discrete Korn's inequality, in general, if $\Gamma_N \neq \emptyset$. A second
ingredient is the discrete inf-sup stability which can already be found in the original paper \cite{ForSou:83} (cf. \cite{For:85} for
the three-dimensional case). The a posteriori error estimator will be based on an approximation to the stress tensor
$\bsigma = 2 \mu \bepsilon (\bu) + p \bI$ in conforming and nonconforming subspaces of $H (\div,\Omega)^2$.

From the discontinuous piecewise linear approximation $\bsigma_h = 2 \mu \bepsilon (\bu_h) + p_h \bI$ obtained from the
solution $( \bu_h , p_h ) \in \bV_h \times Q_h$ of the discrete saddle point problem (\ref{eq:disp_pressure_nonconforming}),
we first reconstruct an $H (\div)$-conforming nonsymmetric stress tensor $\bsigma_h^R$ in the Raviart-Thomas space
(componentwise) $\bSigma_h^R$ of next-to-lowest
order usually denoted by $RT_1$ (see, e.g., \cite[Sect. 2.3.1]{BofBreFor:13}). We will work with the corresponding subspace
$\bSigma_h^R \subset H_{\Gamma_N} (\div,\Omega)^2$, where the normal flux is set to zero on
$\Gamma_N = \partial \Omega \backslash \Gamma_D$.
From now on, we
also assume that the families of triangulations $\cT_h$ are shape-regular to make sure that the constants appearing in our
estimates remain independent of $h$. The diameter of an element $T \in \cT_h$ is then denoted by $h_T$.

We use a stress reconstruction $\bsigma_h^R$ in the next-to-lowest order Raviart-Thomas space $\bSigma_h^R$ which
is constructed by the following procedure which is equivalent to the algorithms in \cite{AinAllBarRan:12} and \cite{Kim:12a}. On
each edge $E \in \cE_h$ we set
\begin{equation}
  \bsigma_h^R \cdot \bn = \{ \! \{ \bsigma_h \cdot \bn \} \! \}_E
  {\colb = \left\{ \begin{array}{lcl}
  \cP_h^{\Gamma_N} \bg & , & E \subset \Gamma_N \: , \\
  \left. \bsigma_h \cdot \bn \right|_{T-} & , & E \subset \Gamma_D \: , \\ 
  \left( \left. \bsigma_h \cdot \bn \right|_{T-} + \left. \bsigma_h \cdot \bn \right|_{T+} \right)/2 & , & \mbox{ otherwise} \: ,
  \end{array} \right.}
  \label{eq:surface_degrees}
\end{equation}
{\colb where $T-$ and $T+$ denote the left and right adjacent triangles to $E$, respectively.}
{\colb The continuity of $\langle \bsigma_h \cdot \bn_E , 1 \rangle_E$ and the identity
$( \div \: \bsigma_h + \bff , 1 )_{L^2 (T)} = 0$ hold for Fortin-Soulie elements, see
\cite[Thm. 1]{ForSou:83}.}
The remaining four degrees of freedom per element are chosen such that, on each element $T \in \cT_h$,
\begin{equation}
  ( \div \: \bsigma_h^R + \cP_h \bff , \widehat{\bq}_h )_{L^2 (T)} = 0
  \label{eq:interior_degrees}
\end{equation}
holds for all $\widehat{\bq}_h$ polynomial of degree 1 satisfying $( \widehat{\bq}_h , 1 )_{L^2 (T)} = 0$. The {\colb same line
of proof as in \cite[Thm. 1]{ForSou:83} for the} conservation
properties fulfilled by the quadratic nonconforming finite elements {\colb leads to }
\begin{equation}
  ( \div \: \bsigma_h^R + \cP_h \bff , 1 )_{L^2 (T)} = 0
  \label{eq:average_equilibration}
\end{equation}
for all $T \in \cT_h$ {\colb (note that $\cP_h \bff$ may be replaced by its average over $T$, cf. \cite[Lemma 1]{Kim:12a}).}
Combined with (\ref{eq:interior_degrees}), $\div \: \bsigma_h^R + \cP_h \bff = \bzero$ holds if
$\bsigma_h^R$ is computed from the nonconforming Galerkin approximation.

The above construction may be divided into elementary substeps using the following decomposition of the space
$\bSigma_h^R$:
$\bSigma_h^R = \bSigma_h^{R,0} \oplus \bSigma_h^{R,1} \oplus \bSigma_h^{R,\Delta}$ with $\bSigma_h^{R,0}$ and
$\bSigma_h^{R,0} \oplus \bSigma_h^{R,1}$ being the corresponding subspaces of the lowest-order Raviart-Thomas
elements $RT_0$ and Brezzi-Douglas-Marini elements $BDM_1$, respectively. In particular,
\begin{equation}
  \begin{split}
    \bSigma_h^{R,1} & = \{ \bsigma_h \in H_{\Gamma_N} (\div,\Omega)^2 : \left. \bsigma_h \right|_T \in P_1 (T)^{2 \times 2}
    \mbox{ for all } T \in \cT_h \: , \\
    & \;\;\;\;\;\;\; \langle \bn_E \cdot (\bsigma_h \cdot \bn_E) , 1 \rangle_E
    = \langle \bt_E \cdot (\bsigma_h \cdot \bn_E) , 1 \rangle_E
    = 0
    \mbox{ for all } E \in \cE_h \} \: , \\
    \bSigma_h^{R,\Delta} & = \{ \bsigma_h \in {\colb \bSigma_h^R} : \left. \div \: \bsigma_h \right|_T \in P_1 (T)^2
    \mbox{ for all } T \in \cT_h \: , \\
    & \;\;\;\;\;\;\; \bsigma_h \cdot \bn_E = \bzero \mbox{ on } E \mbox{ for all } E \in \cE_h \} \: .
  \end{split}
  \label{eq:RT0_BDM1_RT1}
\end{equation}
The above stress reconstruction procedure can be rewritten in the following steps:\\
{\em Step 1.} Compute $\bsigma_h^{R,0} \in \bSigma_h^{R,0}$ such that
$\langle \bn_E \cdot (\bsigma_h^{R,0} \cdot \bn_E) , 1 \rangle_E = \langle \bn_E \cdot (\bsigma_h \cdot \bn_E) , 1 \rangle_E$\\
\phantom{Step 1.}
and $\langle \bt_E \cdot (\bsigma_h^{R,0} \cdot \bn_E) , 1 \rangle_E = \langle \bt_E \cdot (\bsigma_h \cdot \bn_E) , 1 \rangle_E$
holds for all $E \in \cE_h$.\\
{\em Step 2.} Compute $\bsigma_h^{R,1} \in \bSigma_h^{R,1}$ s.t.
$\langle \bn_E \cdot (\bsigma_h^{R,1} \cdot \bn_E) , q_E \rangle_E
= \langle \{ \! \{ \bn_E \cdot (\bsigma_h \cdot \bn_E) \} \! \} , q_E \rangle_E$\\
\phantom{Step 2.}
and $\langle \bt_E \cdot (\bsigma_h^{R,1} \cdot \bn_E) , q_E \rangle_E
= \langle \{ \! \{ \bt_E \cdot (\bsigma_h \cdot \bn_E) \} \! \} , q_E \rangle_E$ holds for all $q_E \in P_1 (E)$\\
\phantom{Step 2.}
with $\langle q_E , 1 \rangle_E = 0$ for all $E \in \cE_h$.\\
{\em Step 3.} Compute $\bsigma_h^{R,\Delta} \in \bSigma_h^{R,\Delta}$ such that
$( \div \: \bsigma_h^{R,\Delta} , \bq_T )_{L^2 (T)} = ( \cP_h \bff , \bq_T )_{L^2 (T)}$ holds
\phantom{Step 3.}
for all $\bq_T \in P_1 (T)^2$ with $( \bq_T , 1 )_{L^2 (\Omega)} = 0$ for all $T \in \cT_h$.\\
Set $\bsigma_h^R = \bsigma_h^{R,0} + \bsigma_h^{R,1} + \bsigma_h^{R,\Delta}$.

\section{A preliminary version of our a posteriori error estimator}

\label{sec-first_estimator}

In this section we present an a posteriori error estimator based on the stress reconstruction $\bsigma_h^R$. It will be
shown to be robust in the incompressible limit $\lambda \rightarrow \infty$ but its reliability bound contains two constants which
are generally not known and may become rather large depending on the shapes of $\Omega$, $\Gamma_D$
and $\Gamma_N$. Therefore, we will modify this a posteriori error estimator in subsequent sections in order to get a guaranteed
upper bound for the error which involves only constants that are at our disposal. The derivation of the more straightforward error
estimator in this section does, however, contribute to the understanding of the situation and already contains some of the crucial
steps of the analysis. One of the unknown constants appearing in the reliability bound is $C_K$, the constant from the Korn
inequality (\ref{eq:discrete_Korn}) which, roughly speaking, becomes big when $\Gamma_D$ is small. The other constant
comes from the following dev-div inequality which was proved under different assumptions in \cite[Prop. 9.1.1]{BofBreFor:13},
\cite{CarDol:98} and \cite{CaiSta:04}:
\begin{equation}
  \| \btau \| \leq C_A \left( \| \bdev \: \btau \| + \| \div \: \btau \| \right)
  \mbox{ for all } \btau \in H_{\Gamma_N} (\div,\Omega)^2 \: ,
  \label{eq:dev-div-inequality}
\end{equation}
where $\bdev$ denotes the trace-free part given by $\bdev \: \btau = \btau - (\tr \: \btau)/2$. Note that the constant $C_A$ may
become big if $\Gamma_N$ is small (cf., for example, \cite[Sect. 2]{MueSta:16}){\colb ; the case $\Gamma_N = \emptyset$
requires an additional constraint $( \tr \: \btau , 1 ) = 0$}.

Our aim is to estimate the error in the energy norm, expressed in terms of $\bu - \bu_h$ and $p - p_h$. Since
(\ref{eq:disp_pressure}) implies $\div \: \bu = p / \lambda$ and since $\div \: \bu_h = p_h / \lambda$ follows from
(\ref{eq:disp_pressure_nonconforming}), the energy norm is given by
\begin{equation}
  \begin{split}
    ||| ( \bu - \bu_h , p - p_h ) ||| & = ( 2 \mu \bepsilon (\bu - \bu_h) + ( p - p_h ) \bI , \bepsilon (\bu - \bu_h) )_h^{1/2} \\
    & = \left( 2 \mu \| \bepsilon (\bu - \bu_h) \|_h^2 + \frac{1}{\lambda} \| p - p_h \|^2 \right)^{1/2} \: .
  \end{split}
  \label{eq:energy_norm}
\end{equation}
Local versions of the energy norm in (\ref{eq:energy_norm}), where the integration is limited to a subset
$\omega \subset \Omega$, will be denoted by $||| ( \: \cdot \: , \: \cdot \: ) |||_\omega$.
The definition of the stress directly leads to
\begin{equation}
  \tr \: \bsigma = 2 \mu \div \: \bu + 2 p = 2 (\mu + \lambda) \div \: \bu \mbox{ and }
  \tr \: \bsigma_h = 2 \mu \div \: \bu_h + 2 p_h = 2 (\mu + \lambda) \div \: \bu_h \: ,
\end{equation}
which implies
\begin{equation}
  \bepsilon (\bu) = \frac{1}{2 \mu} \left( \bsigma - p \bI \right)
  = \frac{1}{2 \mu} \left( \bsigma - \frac{\lambda}{2 (\mu + \lambda)} (\tr \: \bsigma) \bI \right)
  =: \cA \bsigma
  \label{eq:strain_stress_relation}
\end{equation}
and
\begin{equation}
  \bepsilon (\bu_h) = \frac{1}{2 \mu} \left( \bsigma_h - p_h \bI \right)
  = \frac{1}{2 \mu} \left( \bsigma_h - \frac{\lambda}{2 (\mu + \lambda)} (\tr \: \bsigma_h) \bI \right)
  = \cA \bsigma_h \: .
  \label{eq:strain_stress_relation_discrete}
\end{equation}
Note that (\ref{eq:strain_stress_relation}) and (\ref{eq:strain_stress_relation_discrete}) remain valid in the
incompressibe limit $\lambda \rightarrow \infty$, where $\cA$ tends to the projection onto the trace-free part, scaled by
$1/(2 \mu)$. Also note that our stress representation is purely two-dimensional while the true stress in plane-strain
two-dimensional elasticity possesses a three-dimensional component $\sigma_{33} = p$. The following derivation may
equivalently be done based on this three-dimensional stress using a slightly different mapping $\cA_3$ instead of $\cA$.
At the end, however, the result is the same and we therefore stick to the unphysical but simpler two-dimensional stresses.

The inner product $( \: \cdot \: , \: \cdot \: )_\cA := ( \: \cA (\cdot) \: , \: \cdot \: )$ induces a norm $\| \: \cdot \: \|_\cA$
on the divergence-free subspace of $H_{\Gamma_N} (\div,\Omega)^2$ due to
\begin{equation}
  \begin{split}
    \| \btau \|_\cA^2 = ( \cA \btau , \btau )
    & = \frac{1}{2 \mu} ( \btau - \frac{\lambda}{2 (\mu + \lambda)} (\tr \: \btau) \bI , \btau ) \\
    & = \frac{1}{2 \mu} ( \bdev \: \btau , \btau ) + \frac{1}{4 (\mu + \lambda)} ( (\tr \: \btau) \bI , \btau ) \\
    & = \frac{1}{2 \mu} \| \bdev \: \btau \|^2 + \frac{1}{4 (\mu + \lambda)} \| \tr \: \btau \|^2
    \geq \frac{1}{2 \mu} \| \bdev \: \btau \|^2 \: ,
  \end{split}
\end{equation}
which, combined with (\ref{eq:dev-div-inequality}), gives
\begin{equation}
  \sqrt{2 \mu} C_A \| \btau \|_\cA \geq \| \btau \| \mbox{ for all } \btau \in H_{\Gamma_N} (\div,\Omega)^2 \mbox{ with }
  \div \: \btau = \bzero \: ,
  \label{eq:norm_A_equivalence}
\end{equation}
{\colb again assuming $( \tr \: \btau , 1 ) = 0$ if $\Gamma_N = \emptyset$. For the derivation of our preliminary estimator in
this section, however, we need to assume that $\Gamma_N \neq \emptyset$. This restriction will be overcome by the improved
estimator in the next section.} For the element-wise contribution to the energy norm in (\ref{eq:energy_norm}),
\begin{equation}
  \begin{split}
    ||| ( \bu - \bu_h , p - p_h ) |||_T^2
    & = 2 \mu \| \bepsilon (\bu - \bu_h) \|_{L^2 (T)}^2 + \frac{1}{\lambda} \| p - p_h \|_{L^2 (T)}^2 \\
    & = \| 2 \mu \bepsilon (\bu - \bu_h) - (p - p_h) \bI  \|_{\cA,T}^2
  \end{split}
  \label{eq:energy_norm_A_local}
\end{equation}
holds, where $\| \: \cdot \: \|_{\cA,T} = ( \cA (\cdot) , \: \cdot \: )_{L^2 (T)}^{1/2}$ denotes the corresponding element-wise norm.
With the corresponding nonconforming version $\| \: \cdot \: \|_{\cA,h} = ( \cA (\cdot) , \: \cdot \: )_h^{1/2}$ , summing
(\ref{eq:energy_norm_A_local}) over all
elements leads to
\begin{equation}
  \| 2 \mu \bepsilon (\bu - \bu_h) - (p - p_h) \bI  \|_{\cA,h}^2 = ||| ( \bu - \bu_h , p - p_h ) |||^2 \: .
  \label{eq:energy_norm_A}
\end{equation}

Our a posteriori error estimator will be based on $\| \bsigma_h^R - 2 \mu \bepsilon (\bu_h) - p_h \bI  \|_{\cA,h}$, the
difference between post-processed and reconstructed stress and we start the derivation from
this term. Let us assume that (\ref{eq:disp_pressure}) holds with $\bff = \cP_h \bff$ and $\bg = \cP_h^{\Gamma_N} \bg$ since
this approximation can be treated as an oscillation term (see the remarks at the end of this section). Inserting the relation
$\bsigma = 2 \mu \bepsilon (\bu) + p \bI$ which holds for the exact solution, we obtain
\begin{equation}
  \begin{split}
  \| \bsigma_h^R & - 2 \mu \bepsilon (\bu_h) - p_h \bI  \|_{\cA,h}^2
  = \| \bsigma - \bsigma_h^R - 2 \mu \bepsilon (\bu - \bu_h) - (p - p_h) \bI  \|_{\cA,h}^2 \\
  & = \| \bsigma - \bsigma_h^R \|_\cA^2 + \| 2 \mu \bepsilon (\bu - \bu_h) - (p - p_h) \bI  \|_{\cA,h}^2 \\
  & \;\;\; - 2 ( \bsigma - \bsigma_h^R , \cA ( 2 \mu \bepsilon (\bu - \bu_h) - (p - p_h) \bI ) )_h \\
  & = \| \bsigma - \bsigma_h^R \|_\cA^2 + \| 2 \mu \bepsilon (\bu - \bu_h) - (p - p_h) \bI  \|_{\cA,h}^2
  - 2 ( \bsigma - \bsigma_h^R , \bepsilon (\bu - \bu_h) )_h \: ,
  \end{split}
  \label{eq:first_estimator_first}
\end{equation}
where (\ref{eq:strain_stress_relation}) and (\ref{eq:strain_stress_relation_discrete}) were used in the last equality. Our goal is to
estimate the middle term on the right-hand side which by (\ref{eq:energy_norm_A}) coincides with the energy norm. The mixed
term at the end of (\ref{eq:first_estimator_first}) can be rewritten as
\begin{equation}
  \begin{split}
    ( \bsigma - \bsigma_h^R , \: & \bepsilon (\bu - \bu_h) )_h
    = ( \bsigma - \bsigma_h^R , \bnabla (\bu - \bu_h) )_h - ( \bsigma - \bsigma_h^R , \as \: \bnabla (\bu - \bu_h) )_h \\
    & = \sum_{E \in \cE_h} \langle (\bsigma - \bsigma_h^R) \cdot \bn , \llbracket \bu - \bu_h \rrbracket_E \rangle_E
    - ( \bsigma - \bsigma_h^R , \as \: \bnabla (\bu - \bu_h) )_h \\
    & = \sum_{E \in \cE_h} \langle (\bsigma - \bsigma_h^R) \cdot \bn , \llbracket \bu_h^C - \bu_h \rrbracket_E \rangle_E
    - ( \bsigma - \bsigma_h^R , \as \: \bnabla (\bu - \bu_h) )_h \\
    & = ( \bsigma - \bsigma_h^R , \bnabla (\bu_h^C - \bu_h) )_h - ( \bsigma - \bsigma_h^R , \as \: \bnabla (\bu - \bu_h) )_h \\
    & = ( \bsigma - \bsigma_h^R , \bepsilon (\bu_h^C - \bu_h) )_h - ( \bsigma - \bsigma_h^R , \as \: \bnabla (\bu - \bu_h^C) )_h \\
    & = ( \bsigma - \bsigma_h^R , \bepsilon (\bu_h^C - \bu_h) )_h + ( \as \: \bsigma_h^R , \as \: \bnabla (\bu - \bu_h^C) )_h \\
    & = ( \bsigma - \bsigma_h^R , \bepsilon (\bu_h^C - \bu_h) )_h + ( \as \: \bsigma_h^R , \bnabla (\bu - \bu_h^C) )_h
  \end{split}
  \label{eq:first_estimator_second}
\end{equation}
where $\llbracket \: \cdot \: \rrbracket_E$ denotes the jump across $E$ and the properties $\div \: ( \bsigma - \bsigma_h^R) = 0$
{\colb in $\Omega$, $(\bsigma - \bsigma_h^R) \cdot \bn = 0$ on $\Gamma_N$ were}
used. In (\ref{eq:first_estimator_second}), $\bu_h^C \in H_{\Gamma_D}^1 (\Omega)^d$ denotes any conforming
approximation to $\bu_h$. This leads to the upper bound
\begin{align*}
  2 ( \bsigma - \: & \bsigma_h^R , \bepsilon (\bu - \bu_h) )_h \leq 2 \| \bsigma - \bsigma_h^R \| \: \| \bepsilon (\bu_h^C - \bu_h) \|_h
  + 2 \| \as \: \bsigma_h^R \| \: \| \bnabla (\bu - \bu_h^C) \|_h \\
  & \leq 2 \sqrt{2 \mu} C_A \| \bsigma - \bsigma_h^R \|_\cA \| \bepsilon (\bu_h^C - \bu_h) \|_h
  + 2 C_K \| \as \: \bsigma_h^R \| \: \| \bepsilon (\bu - \bu_h^C) \|_h \\
  & \leq \| \bsigma - \bsigma_h^R \|_\cA^2 + 2 \mu C_A^2 \| \bepsilon (\bu_h^C - \bu_h) \|_h^2
  + \frac{2}{\mu} C_K^2 \| \as \: \bsigma_h^R \|^2 + \frac{\mu}{2} \| \bepsilon (\bu - \bu_h^C) \|_h^2 \\
  & \leq \| \bsigma - \bsigma_h^R \|_\cA^2 + (2 C_A^2 + 1) \mu \| \bepsilon (\bu_h^C - \bu_h) \|_h^2
  + \frac{2}{\mu} C_K^2 \| \as \: \bsigma_h^R \|^2 + \mu \| \bepsilon (\bu - \bu_h) \|_h^2 \: .
\end{align*}
Inserting this into (\ref{eq:first_estimator_first}) gives us
\begin{align*}
  \| 2 \mu \bepsilon (\bu - \bu_h) & - (p - p_h) \bI  \|_{\cA,h}^2 \leq \| \bsigma_h^R - 2 \mu \bepsilon (\bu_h) - p_h \bI  \|_{\cA,h}^2 \\
  & + (2 C_A^2 + 1) \mu \| \bepsilon (\bu_h^C - \bu_h) \|_h^2
  + \frac{2}{\mu} C_K^2 \| \as \: \bsigma_h^R \|^2 + \mu \| \bepsilon (\bu - \bu_h) \|_h^2 \: ,
\end{align*}
which, combined with (\ref{eq:energy_norm_A}), implies
\begin{equation}
  \begin{split}
    ||| ( \bu - \bu_h , p - p_h ) |||^2 & \leq 2 \big( \| \bsigma_h^R - 2 \mu \bepsilon (\bu_h) - p_h \bI  \|_{\cA,h}^2 \\
    & \;\;\; + (2 C_A^2 + 1) \mu \| \bepsilon (\bu_h^C - \bu_h) \|_h^2 + \frac{2}{\mu} C_K^2 \| \as \: \bsigma_h^R \|^2 \big) \: .
  \end{split}
  \label{eq:first_estimator_final}
\end{equation}

Finally, the treatment of the right-hand side approximation as an oscillation term is described. To this end, let $( \bu , p )$ and
$( \widetilde{\bu} , \tilde{p} )$ be the solutions of (\ref{eq:disp_pressure}) with right-hand side data $( \bff , \bg )$ and
$( \cP_h \bff , \cP_h^{\Gamma_N} \bg )$, respectively. Taking the difference and inserting $\widetilde{\bu} - \bu$ as test function
leads to
\begin{equation}
  \begin{split}
    2 \mu & \| \bepsilon (\widetilde{\bu} - \bu) \|^2 + \frac{1}{\lambda} \| \tilde{p} - p \|^2
    = ( \cP_h \bff - \bff , \widetilde{\bu} - \bu )
    + \langle \cP_h^{\Gamma_N} \bg - \bg , \widetilde{\bu} - \bu \rangle_{L^2 (\Gamma_N)} \\
    & = ( \cP_h \bff - \bff , \widetilde{\bu} - \bu - \cP_h ( \widetilde{\bu} - \bu ) )
    + \langle \cP_h^{\Gamma_N} \bg - \bg ,
    \widetilde{\bu} - \bu - \cP_h^{\Gamma_N} ( \widetilde{\bu} - \bu ) \rangle_{L^2 (\Gamma_N)} \: .
  \end{split}
  \label{eq:oscillation_term}
\end{equation}
Using the local nature of the projections $\cP_h$ and $\cP_h^{\Gamma_N}$, (\ref{eq:oscillation_term}) implies
\begin{equation}
  \begin{split}
    2 \mu \| \bepsilon (\widetilde{\bu} - \bu) \|^2 + \frac{1}{\lambda} \| \tilde{p} - p & \|^2
    \leq \sum_{T \in \cT_h} \| \bff - \cP_h \bff \|_{L^2 (T)} h_T \| \nabla (\widetilde{\bu} - \bu) \|_{L^2 (T)} \\
    + & \sum_{E \subset \Gamma_N}
    \| \bg - \cP_h^{\Gamma_N} \bg \|_{L^2 (E)} h_E^{1/2} \| \widetilde{\bu} - \bu \|_{H^{1/2} (E)} \\
    \leq & \left( \sum_{T \in \cT_h} h_T^2 \| \bff - \cP_h \bff \|_{L^2 (T)}^2 \right)^{1/2} \| \nabla (\widetilde{\bu} - \bu) \| \\
    + & \left( \sum_{E \subset \Gamma_N} h_E \| \bg - \cP_h^{\Gamma_N} \bg \|_{L^2 (E)}^2 \right)^{1/2}
    \| \widetilde{\bu} - \bu \|_{H^{1/2} (\Gamma_N)} \\
    \leq C \left( \sum_{T \in \cT_h} h_T^2 \| \bff - \cP_h \bff \right. & \|_{L^2 (T)}^2 \left.
    + \sum_{E \subset \Gamma_N} h_E \| \bg - \cP_h^{\Gamma_N} \bg \|_{L^2 (E)}^2 \right)^{1/2}
    \| \bepsilon(\widetilde{\bu} - \bu) \|
  \end{split}
  \label{eq:oscillation_estimate_first}
\end{equation}
with a constant $C$. This proves that
\begin{equation}
  \begin{split}
    ||| ( \widetilde{\bu} - \bu , & \tilde{p} - p ) ||| \\
    & \leq C \left( \sum_{T \in \cT_h} h_T^2 \| \bff - \cP_h \bff \|_{L^2 (T)}^2
    + \sum_{E \subset \Gamma_N} h_E \| \bg - \cP_h^{\Gamma_N} \bg \|_{L^2 (E)}^2 \right)^{1/2}
  \end{split}
  \label{eq:oscillation_estimate}
\end{equation}
and therefore the right-hand side approximation may be treated as an oscillation term.

\section{An improved and guaranteed a posteriori error estimator}

\label{sec-guaranteed_estimator}

We will now construct an improved a posteriori error estimator which avoids the unknown constants $C_A$ and $C_K$ present
in (\ref{eq:first_estimator_final}). To this end, we perform two modifications, one associated with the stress reconstruction
$\bsigma_h^R$ and the other one with the conforming approximation $\bu_h^C$.

The modified stress reconstruction $\bsigma_h^S$ will be constructed in such a way that it satisfies
\begin{equation}
  ( \as \: \bsigma_h^S , \bJ )_{L^2 (T)} = 0 \mbox{ with } \bJ = \begin{pmatrix} 0 & 1 \\ -1 & 0 \end{pmatrix} \mbox{ for all }
  T \in \cT_h \: .
  \label{eq:average_symmetry}
\end{equation}
How this can be achieved will be the topic of section \ref{sec-average_symmetry}. If (\ref{eq:average_symmetry}) is satisfied,
then the corresponding last term in (\ref{eq:first_estimator_second}) can be rewritten as
\begin{equation}
  \begin{split}
    ( \as \: \bsigma_h^S , \nabla (\bu - \bu_h^C) )_h & = \sum_{T \in \cT_h} ( \as \: \bsigma_h^S , \nabla (\bu - \bu_h^C) )_{L^2 (T)}
    \\
    & = \sum_{T \in \cT_h} ( \as \: \bsigma_h^S , \nabla (\bu - \bu_h^C) - \alpha_T \bJ )_{L^2 (T)}
  \end{split}
  \label{eq:improved_antisymmetric_step}
\end{equation}
for any $\alpha_T \in \R$, $T \in \cT_h$. Since for any $\brho_T \in RM (T)$, the space of rigid-body modes on $T$, we have
$\nabla \brho_T = \alpha_T \bJ$ with $\alpha_T \in \R$, (\ref{eq:improved_antisymmetric_step}) leads to
\begin{align*}
  ( \as \: \bsigma_h^S , \nabla (\bu - \bu_h^C) )_h
  & = \sum_{T \in \cT_h} ( \as \: \bsigma_h^S , \nabla (\bu - \bu_h^C - \brho_T ) )_{L^2 (T)} \\
  & \leq \sum_{T \in \cT_h} \| \as \: \bsigma_h^S \|_{L^2 (T)} \: \| \nabla (\bu - \bu_h^C - \brho_T ) \|_{L^2 (T)} \\
  & \leq \| \as \: \bsigma_h^S \| \left( \sum_{T \in \cT_h} \| \nabla (\bu - \bu_h^C - \brho_T ) \|_{L^2 (T)}^2 \right)^{1/2} \: .
\end{align*}
Since $\brho_T \in RM (T)$ is arbitrary, this implies
\begin{equation}
  ( \as \: \bsigma_h^S , \nabla (\bu - \bu_h^C) )_h
  \leq \| \as \: \bsigma_h^S \| \left( \sum_{T \in \cT_h} \inf_{\brho_T \in RM (T)} \| \nabla (\bu - \bu_h^C - \brho_T ) \|_{L^2 (T)}^2
  \right)^{1/2} \: .
  \label{eq:improved_antisymmetric}
\end{equation}
We may now use Korn's inequality of the form
\begin{equation}
  \inf_{\brho_T \in RM (T)} \| \nabla (\bu - \bu_h^C - \brho_T ) \|_{L^2 (T)}
  \leq C_{K,T}^\prime \| \bepsilon (\bu - \bu_h^C) \|_{L^2 (T)}
  \label{eq:Korn_local}
\end{equation}
with a constant $C_{K,T}^\prime$ which depends only on the shape of $T$, in particular, on the smallest interior angle
(see \cite[Sect. 3]{Kim:11b}, \cite{Hor:95} for detailed formulas). From (\ref{eq:improved_antisymmetric}) and (\ref{eq:Korn_local})
we are therefore led to
\begin{equation}
  ( \as \: \bsigma_h^S , \nabla (\bu - \bu_h^C) )_h
  \leq C_K^\prime \| \as \: \bsigma_h^S \| \: \| \bepsilon (\bu - \bu_h^C) \|_h
  \label{eq:improved_antisymmetric_final}
\end{equation}
with $C_K^\prime := \max \{ C_{K,T}^\prime : T \in \cT_h \}$. The constant $C_K^\prime$ is therefore fully computable and,
moreover, of moderate size for shape-regular triangulations. 

The constant $C_A$ may be avoided by enforcing the constraint
\begin{equation}
  ( \div \: \bu_h^C , 1 )_{L^2 (T)} = ( \div \: \bu_h , 1 )_{L^2 (T)} \mbox{ for all }T \in \cT_h \: .
  \label{eq:divergence_constraint}
\end{equation}
We will discuss possible approaches for the construction of such a conforming approximation in section
\ref{sec-divergence_conforming}. If (\ref{eq:divergence_constraint}) is satisfied, then the first term on the right-hand side in
(\ref{eq:first_estimator_second}) can be rewritten as
\begin{equation}
  \begin{split}
    ( \bsigma - \bsigma_h^S , \: & \bepsilon (\bu_h^C - \bu_h) )_h \\
    & = \sum_{T \in \cT_h} \left( ( \bsigma - \bsigma_h^S , \bepsilon (\bu_h^C - \bu_h) )_{L^2 (T)}
    - ( \alpha_T , \div (\bu_h^C - \bu_h) )_{L^2 (T)} \right) \\
    & = \sum_{T \in \cT_h} ( \bsigma - \bsigma_h^S - \alpha_T \bI , \bepsilon (\bu_h^C - \bu_h) )_{L^2 (T)}
  \end{split}
  \label{eq:improved_conforming_step}
\end{equation}
with $\alpha_T \in \R$, $T \in \cT_h$. Since $\alpha_T \in \R$ can be chosen arbitrarily for each $T \in \cT_h$, we obtain
\begin{equation}
  ( \bsigma - \bsigma_h^S , \bepsilon (\bu_h^C - \bu_h) )_h
  \leq \sum_{T \in \cT_h} \inf_{\alpha_T \in \R} \| \bsigma - \bsigma_h^S - \alpha_T \bI \|_{L^2 (T)}
  \| \bepsilon (\bu_h^C - \bu_h) \|_{L^2 (T)} \: .
  \label{eq:improved_conforming_estimate_step}
\end{equation}
Since
\begin{align*}
  \| \bsigma - \bsigma_h^S - \alpha_T \bI \|_{L^2 (T)}^2
  & = \| \bdev (\bsigma - \bsigma_h^S) \|_{L^2 (T)}^2
  + \| \frac{1}{2} \tr (\bsigma - \bsigma_h^S - \alpha_T \bI) \bI \|_{L^2 (T)}^2 \\
  & = \| \bdev (\bsigma - \bsigma_h^S) \|_{L^2 (T)}^2
  + 2 \| \frac{1}{2} \tr (\bsigma - \bsigma_h^S - \alpha_T \bI) \|_{L^2 (T)}^2 \: ,
\end{align*}
the term on the right-hand side of (\ref{eq:improved_conforming_estimate_step}) is minimized, if
$( \tr (\bsigma - \bsigma_h^S - \alpha_T \bI) , 1 )_{L^2 (T)} = 0$ holds. Using the version of the dev-div inequality from
\cite[Prop. 9.1.1]{BofBreFor:13} gives, on each element $T \in \cT_h$,
\begin{equation}
  \begin{split}
    \| \btau \|_{L^2 (T)} \leq C_{A,T}^\prime \| \bdev \: \btau \|_{L^2 (T)} & \mbox{ for all } \btau \in H (\div,T)^2 \\
    & \mbox{ with } \div \: \btau = \bzero \mbox{ and } ( \tr \: \btau , 1 )_{L^2 (T)} = 0
  \end{split}
  \label{eq:dev-div-inequality_local}
\end{equation}
with a constant $C_{A,T}^\prime$ which again only depends on the shape-regularity of the triangulation $\cT_h$. In fact,
these constants coincide with those from (\ref{eq:Korn_local}) which is contained in the following result.

\begin{proposition}
  For a two-dimensional triangulation $\cT_h$, the constants $C_{K,T}^\prime$ in (\ref{eq:Korn_local}) and $C_{A,T}^\prime$ in
  (\ref{eq:dev-div-inequality_local}) are identical.
\end{proposition}

\begin{proof}
  We start from (\ref{eq:dev-div-inequality_local}) and note that $(\tr \: \btau , 1)_{L^2 (T)} = 0$ implies
  \[
    \| \btau \|_{L^2 (\Omega)} = \inf_{\alpha_T \in \R} \| \btau - \alpha_T \bI \|_{L^2 (T)} \: .
  \]
  Moreover $\div \: \btau = \bzero$ implies the existence of $\bpsi \in H^1 (T)$ with $\btau = \bnabla^\perp \bpsi$ 
  (cf. \cite[Theorem I.3.1]{GirRav:86} and therefore (\ref{eq:dev-div-inequality_local}) turns into
  \[
    \inf_{\alpha_T \in \R} \| \bnabla^\perp \bpsi - \alpha_T \bI \|_{L^2 (T)}
    \leq C_{A,T}^\prime \| \bdev \: (\bnabla^\perp \bpsi) \|_{L^2 (T)} \: .
  \]
  Component-wise this has the form
  \begin{align*}
    \inf_{\alpha_T \in \R} & \left\|
    \begin{pmatrix} \partial_2 \psi_1 - \alpha_T & - \partial_1 \psi_1 \\
                             \partial_2 \psi_2 & - \partial_1 \psi_2 - \alpha_T
    \end{pmatrix} \right\|_{L^2 (T)} \\
    & \hspace{2cm} \leq C_{A,T}^\prime \left\|
    \begin{pmatrix} \frac{1}{2} (\partial_2 \psi_1 + \partial_1 \psi_2) & - \partial_1 \psi_1 \\
                             \partial_2 \psi_2 & - \frac{1}{2} (\partial_1 \psi_2 + \partial_2 \psi_1)
    \end{pmatrix} \right\|_{L^2 (T)} \: ,
  \end{align*}
  which may be rewritten as
  \[
    \inf_{\alpha_T \in \R} \| \bnabla \bpsi - \alpha_T \bJ \|_{L^2 (T)}
    \leq C_{A,T}^\prime \| \bepsilon \: (\bnabla \bpsi) \|_{L^2 (T)} \: .
  \]
  This is equivalent to the inequality (\ref{eq:Korn_local}) with $\bpsi = \bu - \bu_h^C$.
\end{proof}

Combining (\ref{eq:improved_conforming_estimate_step}) and (\ref{eq:dev-div-inequality_local}) with $C_{A,T}^\prime$
replaced by $C_{K,T}^\prime$ leads to
\begin{align*}
  ( \bsigma - \bsigma_h^S , \bepsilon (\bu_h^C - \bu_h) )_h
  \leq C_K^\prime \| \bdev (\bsigma - \bsigma_h^S) \| \: \| \bepsilon (\bu_h^C - \bu_h) \|_h \: .
\end{align*}
Combined with (\ref{eq:norm_A_equivalence}), this implies
\begin{equation}
  ( \bsigma - \bsigma_h^S , \bepsilon (\bu_h^C - \bu_h) )_h
  \leq \sqrt{2 \mu} C_K^\prime \| \bsigma - \bsigma_h^S \|_\cA \: \| \bepsilon (\bu_h^C - \bu_h) \|_h \: .
  \label{eq:improved_conforming_estimate}
\end{equation}

Inserting our improved estimates (\ref{eq:improved_antisymmetric_final}) and (\ref{eq:improved_conforming_estimate}) into
(\ref{eq:first_estimator_first}) and (\ref{eq:first_estimator_second}), we arrive at
\begin{align*}
  ||| ( \bu - \bu_h , p - p_h ) |||^2 \leq & \big( \| \bsigma_h^S - 2 \mu \bepsilon (\bu_h) - p_h \bI  \|_{\cA,h}^2
  + 2 \mu (C_K^\prime)^2 \| \bepsilon (\bu_h^C - \bu_h) \|_h^2 \\
  & + 2 \mu \delta \| \bepsilon (\bu - \bu_h^C) \|^2 + \frac{(C_K^\prime)^2}{2 \mu \delta} \| \as \: \bsigma_h^S \|^2 \big) \\
  \leq & \big( \| \bsigma_h^S - 2 \mu \bepsilon (\bu_h) - p_h \bI  \|_{\cA,h}^2
  + 2 \mu \left( (C_K^\prime)^2 + 2 \delta \right) \| \bepsilon (\bu_h^C - \bu_h) \|_h^2 \\
  & + 4 \mu \delta \| \bepsilon (\bu - \bu_h) \|_h^2 + \frac{(C_K^\prime)^2}{2 \mu \delta} \| \as \: \bsigma_h^S \|^2 \big)
\end{align*}
for any $\delta > 0$. Since $2 \mu \| \bepsilon (\bu - \bu_h) \|_h^2 \leq ||| ( \bu - \bu_h , p - p_h ) |||^2$ from the definition in
(\ref{eq:energy_norm}), this finally leads to
\begin{equation}
  \begin{split}
    (1 - 2 \delta) ||| ( \bu - \bu_h , p - p_h ) |||^2 & \leq \big( \| \bsigma_h^S - 2 \mu \bepsilon (\bu_h) - p_h \bI  \|_{\cA,h}^2 \\
    + 2 \mu & \left( (C_K^\prime)^2 + 2 \delta \right) \| \bepsilon (\bu_h^C - \bu_h) \|_h^2
    + \frac{(C_K^\prime)^2}{2 \mu \delta} \| \as \: \bsigma_h^S \|^2 \big)
  \end{split}
  \label{eq:guaranteed_estimator}
\end{equation}
for any $\delta \in ( 0 , 1/2 )$. The choice of $\delta$ may be optimized in dependence of the relative size of the individual error
estimator terms
\begin{equation}
  \eta_h^R = \| \bsigma_h^S - 2 \mu \bepsilon (\bu_h) - p_h \bI  \|_{\cA,h} \: , \:
  \eta_h^C = \sqrt{2 \mu} \| \bepsilon (\bu_h^C - \bu_h) \|_h \: , \:
  \eta_h^S = \| \as \: \bsigma_h^S \| / \sqrt{2 \mu} \: .
  \label{eq:error_estimator_terms}
\end{equation}
Since these three contributions to the error estimator turn out to be of comparable size in our computations, choosing
$\delta = 1/4$ is sufficient for our purpose leading to the following guaranteed reliability result.

\begin{theorem}
  With the error estimator terms $\eta_h^R$, $\eta_h^C$ and $\eta_h^S$ defined in (\ref{eq:error_estimator_terms}), the
  error $( \bu - \bu_h , p - p_h )$, measured in the energy norm defined by (\ref{eq:energy_norm}), satisfies
  \begin{equation}
    ||| ( \bu - \bu_h , p - p_h ) ||| \leq \left( 2 ( \eta_h^R )^2
    + \left ( 2 (C_K^\prime)^2 + 1 \right) ( \eta_h^C )^2
    + 8 (C_K^\prime)^2 ( \eta_h^S )^2 \right)^{1/2} \: .
    \label{eq:guaranteed_estimator_final}
  \end{equation}
  \label{theorem-guaranteed_reliability}
\end{theorem}

\section{Construction of stresses with element-wise symmetry on average}

\label{sec-average_symmetry}

This section provides a possible way for the construction of a modified stress reconstruction with the property
$( \as \: \bsigma_h^S , \bJ )_{L^2 (T)} = 0$ for all $T \in \cT_h$. To this end, we go back to the original stress reconstruction
procedure at the end of section \ref{sec-linear_elasticity_stress}. In order to keep the equilibration property
$\div \: \bsigma_h^R + \cP_h \bff = \bzero$ unaffected, we compute a correction in
\begin{equation}
  \begin{split}
    \bSigma_h^{R,\perp} & = \{ \bsigma_h \in \bSigma_h^{R,0} + \bSigma_h^{R,1} : \div \: \bsigma_h = \bzero \} \\
    & = \{ \bnabla^\perp \bchi_h : \bchi_h \in H_{\Gamma_N}^1 (\Omega)^2 \: , \: \left. \bchi_h \right|_T \in P_2 (T)^2 \}
    =: \bnabla^\perp \bXi_h \: ,
  \end{split}
  \label{eq:div_free_correction}
\end{equation}
{\colb where $\bXi_h$ is the standard conforming piecewise quadratic finite element space with zero boundary conditions on
$\Gamma_N$.} In order to retain the approximation properties of $\bsigma_h^R$, a first attempt would be to compute
$\bsigma_h^{R,\perp} \in \bSigma_h^{R,\perp}$ such that
\begin{equation}
  \begin{split}
  \| \bsigma_h^{R,\perp} \|^2 & \mbox{ is minimized subject to the constraints } \\
  & ( \as \: (\bsigma_h^R + \bsigma_h^{R,\perp}) , \bJ )_{L^2 (T)} = 0 \mbox{ for all } T \in \cT_h \: .
  \end{split}
  \label{eq:overdetermined_system_global}
\end{equation}
Inserting $\bsigma_h^{R,\perp} = \bnabla^\perp \bchi_h$ with $\bchi_h \in \bXi_h$,
(\ref{eq:overdetermined_system_global}) turns out to be equivalent to
\begin{equation}
  \begin{split}
  \| \bnabla^\perp \bchi_h \|^2 & \rightarrow \min ! \mbox{ subject to the constraints } \\
  & ( \div \: \bchi_h , 1 )_{L^2 (T)} = ( \as \: \bsigma_h^R , \bJ )_{L^2 (T)} \mbox{ for all } T \in \cT_h \: .
  \end{split}
  \label{eq:overdetermined_system_representation_global}
\end{equation}
The solution $\bchi_h^\perp \in \bXi_h$ of (\ref{eq:overdetermined_system_representation_global}) is uniquely determined by the
associated KKT conditions
\begin{equation}
  \begin{split}
    ( \bnabla^\perp \bchi_h^\perp , \bnabla^\perp \bxi_h ) - ( \div \: \bxi_h , \nu_h ) & = 0 \\
    - ( \div \: \bchi_h^\perp , \rho ) & = - ( \as \: \bsigma_h^R , \rho_h \bJ )
  \end{split}
  \label{eq:KKT_system_global}
\end{equation}
for all $\bxi_h \in \bXi_h$ and $\rho_h \in Z_h$, where $Z_h$ denotes the space of scalar piecewise constant functions with
respect to $\cT_h$. In (\ref{eq:KKT_system_global}), $\nu_h \in Z_h$ plays the role of a Lagrange multiplier.
{\colb The discrete inf-sup stability of the $P_2-P_0$ combination in two dimensions (cf. \cite[Sect. 8.4.3]{BofBreFor:13}) and
the coercivity in $H_{\Gamma_N}^1 (\Omega)$ leads to the well-posedness of (\ref{eq:KKT_system_global}).}
The modified stress reconstruction is given by $\bsigma_h^S = \bsigma_h^R + \bnabla^\perp \bchi_h$.

{\colb It would be desirable to replace the global minimization problem (\ref{eq:overdetermined_system_representation_global})
by a set of local ones like it will be done later in the next section for the estimation of the non-conformity error. The main
incentive for such an approach would be the control of the correction $\bsigma_h^{R,\perp}$ on an element $T$ by the
right-hand side $( \as \: \bsigma_h^R , \bJ )$ in a neighborhood of $T$. Unfortunately, this is not possible, in general, since the
following situation may occur in principle: On one element $T^\ast$ away from the boundary,
$( \as \: \bsigma_h^R , \bJ )_{L^2 (T^\ast)}$ does not vanish while $( \as \: \bsigma_h^R , \bJ )_{L^2 (T)} = 0$ for all
$T \in \cT_h \backslash \{ T^\ast \}$. Then, it is not possible to find an admissible $\bchi_h$ such that its support
$\mbox{supp} \: \bchi_h$ is contained in a subset $\omega \subset \Omega$ that does not touch the boundary,
$\partial \omega \cap \partial \Omega = \emptyset$. This is due to the fact that, in this case,
\begin{equation}
  \begin{split}
    \sum_{T \in \cT_h} ( \div \: \bchi_h , 1 )_{L^2 (T)} & = ( \div \: \bchi_h , 1 )_{L^2 (\omega)}
    = \langle \bchi_h \cdot \bn , 1 \rangle_{\partial \omega} = 0 \\
    & \neq ( \as \: \bsigma_h^R , \bJ )_{L^2 (T^\ast)} = \sum_{T \in \cT_h} ( \as \: \bsigma_h^R , \bJ )_{L^2 (T)}
  \end{split}
\end{equation}
would hold. From the computational point of view, (\ref{eq:overdetermined_system_representation_global}) constitutes a
saddle point problem which requires much less effort to solve than the original one (\ref{eq:disp_pressure_nonconforming}).
}

The following result gives an upper bound for the correction $\bnabla^\perp \bchi_h$ which will later be used for showing
{\colb global} efficiency of the error estimator.

{\colb
\begin{proposition}
  The correction $\bchi_h^\perp \in \bXi_h$ defined by
  (\ref{eq:overdetermined_system_representation_global}) satisfies
  \begin{equation}
    \| \bnabla^\perp \bchi_h^\perp \| \leq C \| \bsigma - \bsigma_h^R \| \: ,
  \end{equation}
  where the constant $C$ depends only on the shape regularity of $\cT_h$.
  \label{prop-average_symmetry_correction}
\end{proposition}

\begin{proof}
  Since $\bchi_h^\perp$ is a solution of the KKT system (\ref{eq:KKT_system_global}), we obtain
  \begin{equation}
    \| \bnabla^\perp \bchi_h^\perp \|^2
    = ( \bnabla^\perp \bchi_h^\perp , \bnabla^\perp \bchi_h^\perp )
    = ( \div \: \bchi_h^\perp , \nu_h ) = ( \as \: \bsigma_h^R , \nu_h \bJ ) \: .
    \label{eq:correction_relation}
  \end{equation}
  From the well-posedness of (\ref{eq:KKT_system_global}) we obtain
  \begin{equation}
    \| \nu_h \| \leq \frac{C^2}{\sqrt{2}} \| \as \: \bsigma_h^R \|
    \label{eq:inf_sup_estimate}
  \end{equation}
  with a constant $C$ (cf. \cite[Theorem 5.2.1]{BofBreFor:13}). A combination of (\ref{eq:correction_relation}) and
  (\ref{eq:inf_sup_estimate}) implies
  \begin{equation}
    \| \bnabla^\perp \bchi_h^\perp \|^2 \leq \| \as \: \bsigma_h^R \| \: \| \nu_h \bJ \|
    \leq \sqrt{2} \| \as \: \bsigma_h^R \| \: \| \nu_h \| \leq C^2 \| \as \: \bsigma_h^R \|^2 \: .
    \label{eq:correction_estimate}
  \end{equation}
  Since $\as \: \bsigma = \bzero$, we obtain
  \begin{equation}
    \| \bnabla^\perp \bchi_h^\perp \| \leq \widetilde{C} \| \as \: (\bsigma - \bsigma_h^R) \|
    \leq \widetilde{C} \| \bsigma - \bsigma_h^R \| \: ,
    \label{eq:correction_estimate_final}
  \end{equation}
  where we used the fact that $| \as \: (\bsigma - \bsigma_h^R) | \leq | \bsigma - \bsigma_h^R |$ holds pointwise.
\end{proof}
}

\section{Distance to divergence-constrained conformity}

\label{sec-divergence_conforming}

Our point of departure for the construction of a divergence-constrained conforming approximation is the minimization problem
\begin{equation}
  \begin{split}
    \| \bnabla \bu_h^C - \bnabla \bu_h \|_h^2 & \rightarrow \min ! \mbox{ subject to the constraints } \\
    & ( \div \: \bu_h^C , 1 )_{L^2 (T)} = ( \div \: \bu_h , 1 )_{L^2 (T)} \mbox{ for all } T \in \cT_h
  \end{split}
  \label{eq:conforming_approximation_global}
\end{equation}
among all $\bu_h^C \in \bV_h^C$, the subspace of conforming piecewise quadratic functions. The
solution of this global minimization problem can be replaced by local ones based on the partition of unity
\begin{equation}
  1 \equiv \sum_{z \in \cV_h^\prime} \phi_z {\colb \mbox{ on } \Omega}
  \label{eq:partition_of_unity_D}
\end{equation}
with respect to $\cV_h^\prime = \{ z \in \cV_h : z \notin \Gamma_D \}$. {\colb Here, $\cV_h$ denotes the set of vertices of the
triangulation and $\phi_z$, $z \in \cV_h^\prime$ are continuous piecewise linear functions with support restricted to
\begin{equation}
  \omega_z := \bigcup \{ T \in \cT_h : z \mbox{ is a vertex of } T \} \: .
  \label{eq:vertex_patch}
\end{equation}
For the partition of unity (\ref{eq:partition_of_unity_D}), the standard pyramid basis functions
need to be extended for all vertices $z \in \cV_h^\prime$ adjacent to a boundary vertex on $\Gamma_D$ such that it is constant
along the connecting edge. This requires that the triangulation $\cT_h$ is such that each vertex on $\Gamma_D$ is connected
by an interior edge.}
From the decomposition
\begin{equation}
  \bu_h = \sum_{z \in \cV_h^\prime} \bu_h \phi_z =: \sum_{z \in \cV_h^\prime} \bu_{h,z}
  \label{eq:sum_nonconforming}
\end{equation}
we are led, for each $z \in \cV_h^\prime$, to the problem
\begin{equation}
  \begin{split}
    \| \bnabla \bu_{h,z}^C - \bnabla \bu_{h,z} \|_h^2 & \rightarrow \min ! \mbox{ subject to the constraints } \\
    & ( \div \: \bu_{h,z}^C , 1 )_{L^2 (T)} = ( \div \: \bu_{h,z} , 1 )_{L^2 (T)} \mbox{ for all } T \subset \omega_z
  \end{split}
  \label{eq:conforming_approximation_local}
\end{equation}
among all $\bu_{h,z}^C \in \widehat{\bV}_{h,z}^C$, where $\widehat{\bV}_{h,z}^C \subset H^1 (\omega_z)$
{\colb may be any space of conforming finite elements vanishing on all edges not adjacent to $z$. The compatibility condition
for the constraint in (\ref{eq:conforming_approximation_local}) is satisfied since $\bu_{h,z}$ and $\bu_{h,z}^C$ both vanish on
$\omega_z$.} Since $\bu_{h,z} = \bu_h \Phi_z$ is piecewise cubic, using conforming elements of polynomial
degree 3 are used for $\widehat{\bV}_{h,z}^C$ in order to {\colb secure the optimal approximation order}.
For each $z \in \cV_h^\prime$, the
solution $\bu_{h,z}^C \in \widehat{\bV}_{h,z}^C$ of the minimization problem (\ref{eq:conforming_approximation_local}) is
obtained from the KKT system
\begin{equation}
  \begin{split}
  ( \bnabla \bu_{h,z}^C , \bnabla \bv_{h,z}^C )_h - ( \div \: \bv_{h,z}^C , \nu_{h,z} ) & = ( \bnabla \bu_{h,z} , \bnabla \bv_{h,z}^C )_h
  \\
  - ( \div \: \bu_{h,z}^C , \rho_{h,z} ) & = - ( \div \: \bu_{h,z} , \rho_{h,z} )_h
  \end{split}
  \label{eq:KKT_system_conforming}
\end{equation}
for all $\bv_{h,z}^C \in \widehat{\bV}_{h,z}^C$ and $\rho_{h,z} \in Z_{h,z}$. These local saddle-point problems for
$\bu_{h,z}^C \in \widehat{\bV}_{h,z}^C$ and $\nu_{h,z} \in Z_{h,z}$ are again well-posed
{\colb due to the inf-sup stability of these combinations of finite element spaces.}
For the conforming approximation
\begin{equation}
  \bu_h^C = \sum_{z \in \cV_h^{\prime\prime}} \bu_{h,z}^C \in \widehat{\bV}_h^C \: ,
  \label{eq:sum_conforming}
\end{equation}
$\widehat{\bV}_h^C \subset H_{\Gamma_D}^1 (\Omega)$ being the piecewise cubic finite element space, one obtains,
for each $T \in \cT_h$,
\[
  ( \div \: \bu_h^C , 1 )_{L^2 (T)} = \sum_{z \in \cV_h^\prime} ( \div \: \bu_{h,z}^C , 1 )_{L^2 (T)}
  = \sum_{z \in \cV_h^\prime} ( \div \: \bu_{h,z} , 1 )_{L^2 (T)} = ( \div \: \bu_h , 1 )_{L^2 (T)} \: ,
\]
i.e., it satisfies the constraint in (\ref{eq:conforming_approximation_global}).

\begin{proposition}
  The conforming approximation $\bu_h^C \in \widehat{\bV}_h^C$ defined by (\ref{eq:KKT_system_conforming}) and
  (\ref{eq:sum_conforming}) satisfies
  \begin{equation}
    \| \bnabla \bu_h^C - \bnabla \bu_h \|_{L^2 (T)} \leq C \| \bnabla \bu - \bnabla \bu_h \|_{L^2 (\omega_T)} \: ,
  \end{equation}
  where the constant $C$ depends only on the shape regularity of $\cT_h$.
  \label{prop-distance_to_conforming}
\end{proposition}

\begin{proof}
  Since $\bu_{h,z}^C \in \widehat{\bV}_{h,z}^C$ solves (\ref{eq:conforming_approximation_local}), a simple scaling argument
  gives us
  \begin{equation}
    \| \bnabla \bu_{h,z}^C - \bnabla \bu_{h,z} \|_{L^2 (\omega_z)}^2 \leq \tilde{C}
    \sum_{E \subset \omega_z} h_E^{-1} \| \llbracket \bu_{h,z} \rrbracket_E \|_{L^2 (E)}^2
    \label{eq:vertex_patch_equivalence}
  \end{equation}
  for all $z \in \cV_h^\prime$ (note that the right-hand side being zero implies $\bu_{h,z} \in \widehat{\bV}_{h,z}^C$
  and therefore the left-hand side also vanishes). From (\ref{eq:sum_nonconforming}) and (\ref{eq:sum_conforming}) we get
  \begin{equation}
    \begin{split}
      \| \bnabla & \bu_h^C - \bnabla \bu_h \|_{L^2 (T)}^2 \\
      & \leq \sum_{z \in \cV_h^\prime \cap T} \| \bnabla \bu_{h,z}^C - \bnabla \bu_{h,z} \|_{L^2 (T)}^2
      \leq \sum_{z \in \cV_h^\prime \cap T} \| \bnabla \bu_{h,z}^C - \bnabla \bu_{h,z} \|_{L^2 (\omega_z)}^2 \\
      & \leq \tilde{C} \sum_{z \in \cV_h^\prime \cap T} \sum_{E \subset \omega_z} h_E^{-1}
      \| \llbracket \bu_{h,z} \rrbracket_E \|_{L^2 (E)}^2 \\
      & = \tilde{C} \sum_{z \in \cV_h^\prime \cap T} \sum_{E \subset \omega_z} h_E^{-1}
      \| \llbracket \bu_h \rrbracket_E \phi_z \|_{L^2 (E)}^2 \\
      & \leq \tilde{C} \sum_{z \in \cV_h^\prime \cap T} \sum_{E \subset \omega_z} h_E^{-1}
      \| \llbracket \bu_h \rrbracket_E \|_{L^2 (E)}^2
      \leq 2 \tilde{C} \sum_{E \subset \omega_T} h_E^{-1} \| \llbracket \bu_h \rrbracket_E \|_{L^2 (E)}^2 \: .
    \end{split}
    \label{eq:gradient_to_jump}
  \end{equation}
  Since $\langle \llbracket \bu_h \rrbracket_E , 1 \rangle_E {\colb = 0}$ is satisfied for all $E \in \cE_h$, the same line of reasoning
  as in \cite[Theorem 10]{AchBerCoq:03} (cf. \cite[Sect. 6]{HanSteVoh:12}) implies that, for all $T \in \cT_h$,
  \begin{equation}
    \sum_{E \subset \omega_T} h_E^{-1} \| \llbracket \bu_h \rrbracket_E \|_{L^2 (E)}^2
    \leq \widehat{C} \| \bnabla \bu - \bnabla \bu_h \|_{L^2 (\omega_T)}
    \label{eq:jump_to_gradient}
  \end{equation}
  holds. Combining (\ref{eq:gradient_to_jump}) and (\ref{eq:jump_to_gradient}) finishes the proof.
\end{proof}

\section{Global Efficiency of the Improved Estimator}

\label{sec-efficiency}

Efficiency of the error estimator is shown if all three terms in (\ref{eq:error_estimator_terms}), $\eta_h^R$, $\eta_h^C$ and
$\eta_h^S$, can be bounded by the energy norm of the error multiplied with a constant which remains bounded in the
incompressible limit.

Using the definition of $\cA$ in (\ref{eq:strain_stress_relation}) {\colb and (\ref{eq:norm_A_equivalence}), Proposition
\ref{prop-average_symmetry_correction} implies
\begin{equation}
  \sqrt{2 \mu} \| \bsigma_h^S - \bsigma_h^R \|_{\cA,h} \leq \| \bsigma_h^S - \bsigma_h^R \|
  \leq C \| \bsigma - \bsigma_h^R \| \leq \sqrt{2 \mu} C_A C \| \bsigma - \bsigma_h^R \|_{\cA,h} \: .
\end{equation}
}
The first estimator term $\eta_h^R$ can therefore be bounded in the form
{\colb
\begin{equation}
  \begin{split}
    \eta_h^R & = \| \bsigma_h^S - 2 \mu \bepsilon (\bu_h) - p_h \bI  \|_{\cA,h} \\
    & \leq \| \bsigma_h^R - 2 \mu \bepsilon (\bu_h) - p_h \bI  \|_{\cA,h} + \| \bsigma_h^S - \bsigma_h^R \|_{\cA,h} \\
    & \leq \| \bsigma_h^R - 2 \mu \bepsilon (\bu_h) - p_h \bI  \|_{\cA,h} + C_A C \| \bsigma - \bsigma_h^R \|_{\cA,h} \\
    & \leq (1 + C_A C) \| \bsigma_h^R - 2 \mu \bepsilon (\bu_h) - p_h \bI \|_{\cA,h}
    + C_A C \| \bsigma - 2 \mu \bepsilon (\bu_h) - p_h \bI \|_{\cA,h} \\
    & = (1 + C_A C) \| \bsigma_h^R - 2 \mu \bepsilon (\bu_h) - p_h \bI \|_{\cA,h} + C_A C ||| (\bu - \bu_h , p - p_h) ||| \: ,
  \end{split}
  \label{eq:local_efficiency_first}
\end{equation}
where (\ref{eq:energy_norm_A}) was used in the last equality.}

The first term on the right-hand side in (\ref{eq:local_efficiency_first}) can be treated by the following lemma.
We will use the notation $a \lesssim b$ to indicate that $a$ is bounded
by $b$ times a constant that is independent of the Lam\'e parameter $\lambda$.

{\colb
\begin{lemma}
  The stress reconstruction computed by the algorithm at the end of Section \ref{sec-linear_elasticity_stress} satisfies
  \begin{equation}
    \| \bsigma_h^R - 2 \mu \bepsilon (\bu_h) - p_h \bI \|_{A,h}
    \lesssim ||| ( \bu - \bu_h , p - p_h ) |||
    + \sum_{T \in \cT_h} h_T^2 \| \bnabla \bff \|_{L^2 (T)} \: .
  \end{equation}
\end{lemma}
}

\begin{proof}
  The first step consists in the observation that for all functions $\btau_h$ which are element-wise of next-to-lowest order
  Raviart-Thomas type,
  \begin{equation}
    \begin{split}
      \| \btau_h \|_{L^2 (T)}^2 & \lesssim \sum_{E \subset \partial T} h_E
      \left( \| \{ \! \{ \btau_h \cdot \bn \} \! \}_E \|_{L^2 (E)}^2 + \| \llbracket \btau_h \cdot \bn \rrbracket_E \|_{L^2 (E)}^2 \right) \\
      & + h_T^2 \| \div \: \btau_h \|_{L^2 (T)}^2
    \end{split}
    \label{eq:norm_equivalence_patch}
  \end{equation}
  holds. This can be shown in the usual way using a scaling argument and the finite dimension of the considered space.
  Applying (\ref{eq:norm_equivalence_patch}) to $\btau_h = \bsigma_h^R - 2 \mu \bepsilon (\bu_h) - p_h \bI$, the right-hand
  side simplifies since, due to (\ref{eq:surface_degrees}) and (\ref{eq:interior_degrees}),
  \begin{equation}
    \begin{split}
      \{ \! \{ ( \bsigma_h^R - 2 \mu \bepsilon (\bu_h) - p_h \bI ) \cdot \bn \} \! \}_E & = \bzero \hspace{1.4cm}
      \mbox{ for all } E \in \cE_h \: , \\
      \left. \div ( \bsigma_h^R - 2 \mu \bepsilon (\bu_h) - p_h \bI ) \right|_T & = \cP_h^0 \bff - \cP_h \bff \mbox{ for all } T \in \cT_h
    \end{split}
  \end{equation}
  holds, where $\cP_h^0$ denotes the $L^2 (\Omega)$-orthogonal projection onto piecewise constant functions.
  {\colb The identity $\div (2 \mu \bepsilon (\bu_h) + p_h \bI) = - \cP_h^0 \bff$ follows from \cite[Thm. 1]{ForSou:83}.}
  From (\ref{eq:norm_equivalence_patch}) we therefore get
  \begin{equation}
    \begin{split}
      \| \bsigma_h^R & - 2 \mu \bepsilon (\bu_h) - p_h \bI \|_{L^2 (T)}^2 \\
      & \lesssim \sum_{E \subset \partial T} h_E
      \| \llbracket ( \bsigma_h^R - 2 \mu \bepsilon (\bu_h) - p_h \bI ) \cdot \bn \rrbracket_E \|_{L^2 (E)}^2
      + h_T^2 \| \cP_h^0 \bff - \cP_h \bff \|_{L^2 (T)}^2 \\
      & \lesssim \sum_{E \subset \partial T} h_E
      \| \llbracket ( 2 \mu \bepsilon (\bu_h) + p_h \bI ) \cdot \bn \rrbracket_E \|_{L^2 (E)}^2
      + h_T^4 \| {\colb \nabla} \bff \|_{L^2 (T)}^2 \: ,
    \end{split}
    \label{eq:norm_equivalence_patch_applied}
  \end{equation}
  {\colb where the approximation estimate
  \begin{equation}
    \| \cP_h^0 \bff - \cP_h \bff \|_{L^2 (T)}^2 = \| \bff - \cP_h^0 \bff \|_{L^2 (T)}^2 - \| \bff - \cP_h \bff \|_{L^2 (T)}^2
    \lesssim h_T^2 \| \bnabla \bff \|_{L^2 (T)}^2
  \end{equation}
  for the $L^2 (T)$-orthogonal projections $\cP_h^0$ and $\cP_h$ onto polynomials of degree 0 and 1 was used.}
  Arguing along the same lines as in \cite[Theorem 6]{AchBerCoq:03} one gets
  \begin{equation}
    \begin{split}
      h_E \| \llbracket ( 2 \mu & \bepsilon (\bu_h) + p_h \bI ) \cdot \bn \rrbracket_E \|_{L^2 (E)}^2 \\
      & \lesssim \| \bsigma - 2 \mu \bepsilon (\bu_h) - p_h \bI \|_{L^2 (\omega_E)}^2
      {\colb + h_T^2 \| \div (\bsigma - 2 \mu \bepsilon (\bu_h) - p_h \bI) \|_{L^2 (\omega_E)}^2} \\
      & \lesssim \| \bepsilon (\bu - \bu_h) \|_{L^2 (\omega_E)}^2 + \| p - p_h \|_{L^2 (\omega_E)}^2
      {\colb + h_T^2 \| \bff - \cP_h^0 \bff \|_{L^2 (\omega_E)}^2} \: ,
    \end{split}
  \end{equation}
  where $\omega_E$ denotes the union of the two elements adjacent to $E$.
  Summing over all $T$
  {\colb leads to
  \begin{equation}
    \| \bsigma_h^R - 2 \mu \bepsilon (\bu_h) - p_h \bI \|^2
    \lesssim \| \bepsilon (\bu - \bu_h) \|_h^2 + \| p - p_h \|^2 + \sum_{T \in \cT_h} h_T^4 \| \bnabla \bff \|_{L^2 (T)}^2
  \end{equation}
  which} finishes the proof.
\end{proof}

For the second term in (\ref{eq:error_estimator_terms}) we may use Proposition \ref{prop-distance_to_conforming} to get
\begin{equation}
  \begin{split}
    \eta_h^C & \lesssim \| \bepsilon (\bu_h^C - \bu_h) \| \leq \| \nabla (\bu_h^C - \bu_h) \| \\
    & \lesssim \| \nabla (\bu - \bu_h) \| \lesssim \| \bepsilon (\bu - \bu_h) \| \lesssim ||| ( \bu - \bu_h , p - p_h ) ||| \: .
  \end{split}
  \label{eq:local_efficiency_second}
\end{equation}
Finally, the third term in (\ref{eq:error_estimator_terms}) satisfies
\begin{equation}
  \begin{split}
    \eta_h^S = \frac{1}{\sqrt{2 \mu}} \| \as \: \bsigma_h^S \|
    = \| \as \: \bsigma_h^S \|_{\cA,h} & = \| \as \: ( \bsigma_h^S - 2 \mu \bepsilon (\bu_h) - p_h \bI ) \|_{\cA,h} \\
    & \leq \| \bsigma_h^S - 2 \mu \bepsilon (\bu_h) - p_h \bI \|_{\cA,h} = \eta_h^R \: .
  \end{split}
  \label{eq:local_efficiency_third}
\end{equation}

We summarize the {\colb global} efficiency result in the following theorem.

\begin{theorem}
  The error estimator terms $\eta_h^R$, $\eta_h^C$ and $\eta_h^S$ defined in (\ref{eq:error_estimator_terms}) satisfy
  \begin{equation}
    \eta_h^R + \eta_h^C + \eta_h^S \lesssim ||| ( \bu - \bu_h , p - p_h ) |||
    + \sum_{T \in \cT_h} h_T^2 \| {\colb \bnabla} \bff \|_{L^2 (T)} \: .
     \label{eq:local_efficiency}
  \end{equation}
  \label{theorem-efficiency}
\end{theorem}

{\colb
The last term in (\ref{eq:local_efficiency}) is of the same order as the approximation error is expected to decrease in the ideal
case for the finite element spaces studied in this paper. It is, however, not an oscillation term of higher order. Nevertheless,
\ref{theorem-efficiency} implies that the error estimator decreases proportionally to the approximation error.
}

\section{Computational Results}

\label{sec-computational}

\begin{figure}[!htb]
  \hspace*{.2cm}\includegraphics[scale=0.47]{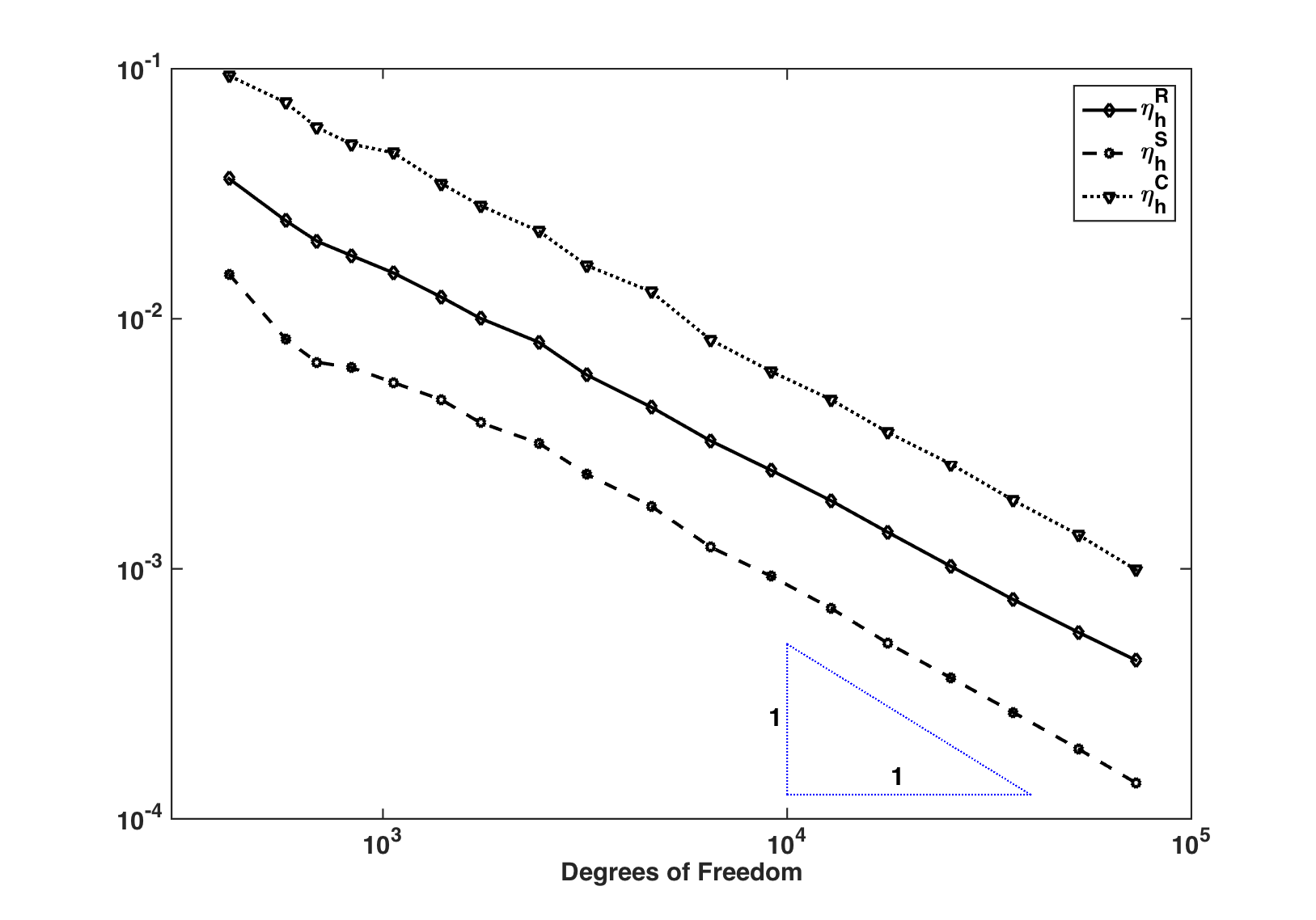}
  \caption{Adaptive finite element convergence: $\nu = 0.29$}
  \label{fig-fe_convergence_cook29}
\end{figure}

\begin{figure}[!htb]
  \hspace*{.2cm}\includegraphics[scale=0.47]{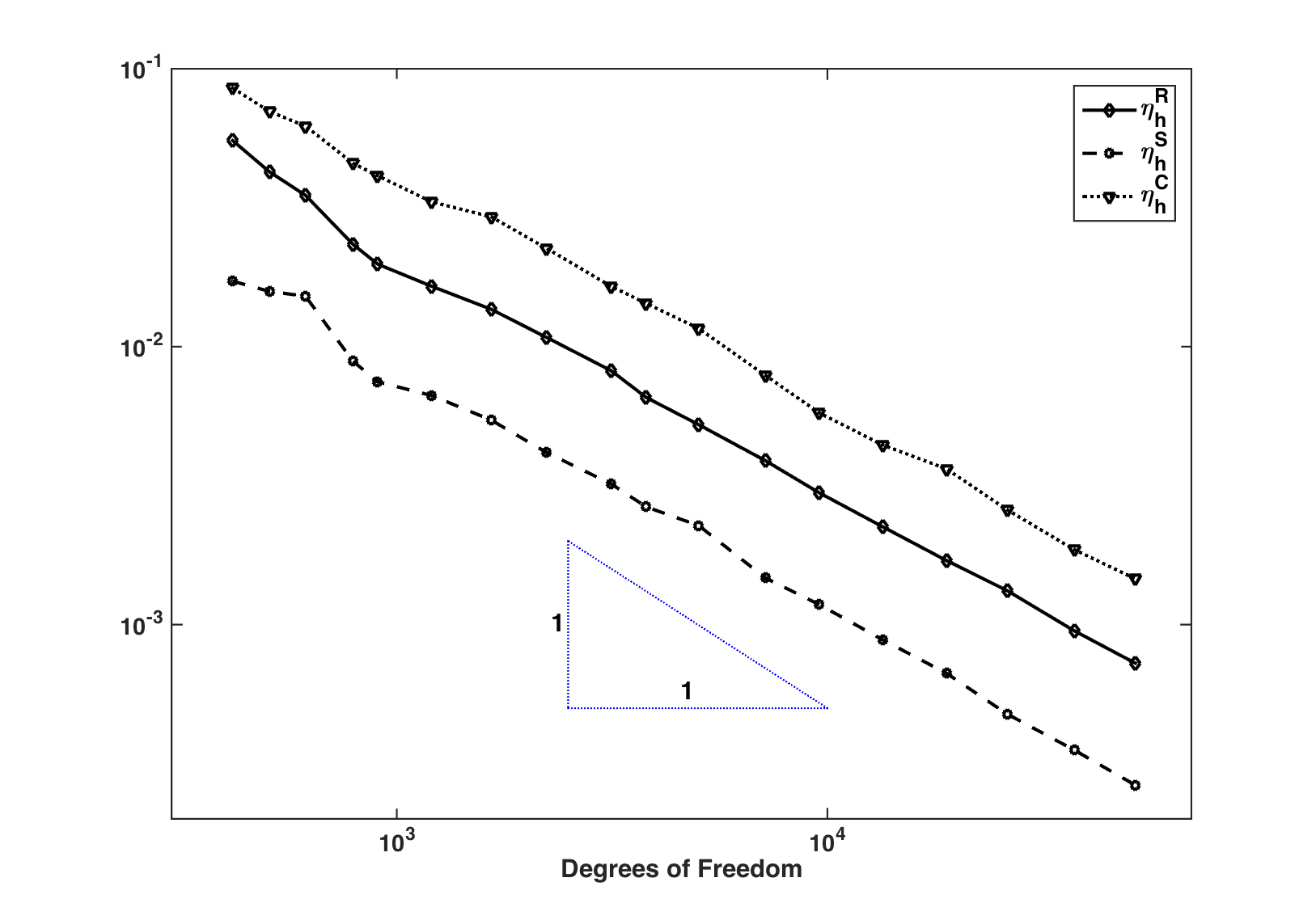}
  \caption{Adaptive finite element convergence: $\nu = 0.49$}
  \label{fig-fe_convergence_cook49}
\end{figure}

\begin{figure}[!htb]
  \hspace*{.2cm}\includegraphics[scale=0.4]{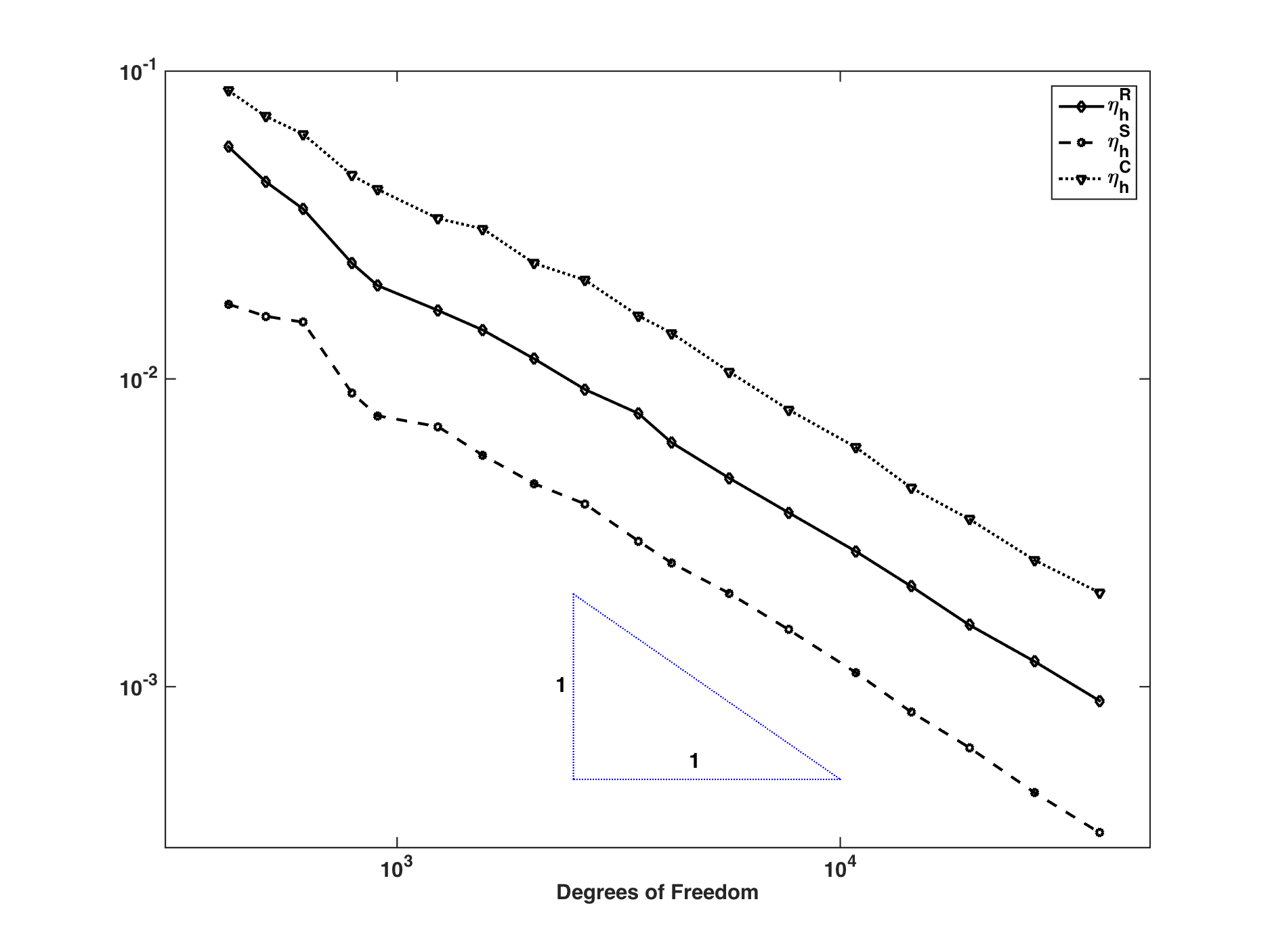}
  \caption{Adaptive finite element convergence: $\nu = 0.5$}
  \label{fig-fe_convergence_cook5}
\end{figure}

This section presents computational results with the a posteriori error estimator studied in the previous sections.
As an example, Cook's
membrane is considered which consists of the quadrilateral domain $\Omega \in \R^2$ with corners $(0,0)$, $(0.48,0.44)$,
$(0.48,0.6)$ and $(0,0.44)$, where $\Gamma_D$ coincides with the left boundary segment. The prescribed surface traction
forces on $\Gamma_N$ are $\bg = \bzero$ on the upper and lower boundary segments and $\bg = (0,1)$ on the right.
Starting from an initial triangulation with 44 elements, 17 adaptive refinement steps are performed based on the equilibration
strategy, where a subset $\widetilde{\cT}_h \subset \cT_h$ of elements is refined such that
\begin{equation}
  \left( \sum_{T \in \tilde{\cT}_h} \eta_T^2 \right)^{1/2} \geq \theta \left( \sum_{T \in \cT_h} \eta_T^2 \right)^{1/2}
  \label{eq:equilibration}
\end{equation}
holds with $\theta = 0.5$ (cf. \cite[Sect. 2.1]{Ver:13}). Figures \ref{fig-fe_convergence_cook29},
\ref{fig-fe_convergence_cook49} and \ref{fig-fe_convergence_cook5} show the convergence behavior in terms of the error
estimator for Poisson ratios $\nu = 0.29$ (compressible case), $\nu = 0.49$ (nearly incompressible case) and $\nu = 0.5$
(incompressible case). Since the Poisson ratio is related to the Lam\'e parameters by $2 \mu \nu = \lambda (1 - 2 \nu)$ and
since $\mu$ is set to 1 in our computations, this leads to the values $\lambda = 1.381$, $\lambda = 49$ and $\lambda = \infty$
in the three examples.
The solid line (always in the middle) represents the estimator term $\eta_h^R$, the dashed line below stands for $\eta_h^S$
measuring the skew-symmetric part and the dotted line shows the values for $\eta_h^C$, the distance to the conforming space.
In all cases, the optimal convergence behavior $\eta_h^\Box \sim N_h^{-1}$, if $N_h$ denotes the number of unknowns, is
observed. {\colb For the investigation of the effectivity of error estimators of the type presented in this paper we refer to
\cite{AinAllBarRan:12,Kim:12a} where the case of the Stokes equations with $\Gamma_D = \partial \Omega$ is treated. The fact
that the estimator term $\eta_h^S$ measuring the symmetry is dominated by the other two contributions
$\eta_h^R$ and $\eta_h^C$ suggests that the effectivity indices are comparable to those reported in these references.}

\begin{figure}[!htb]
  \hspace*{.2cm}\includegraphics[scale=0.22]{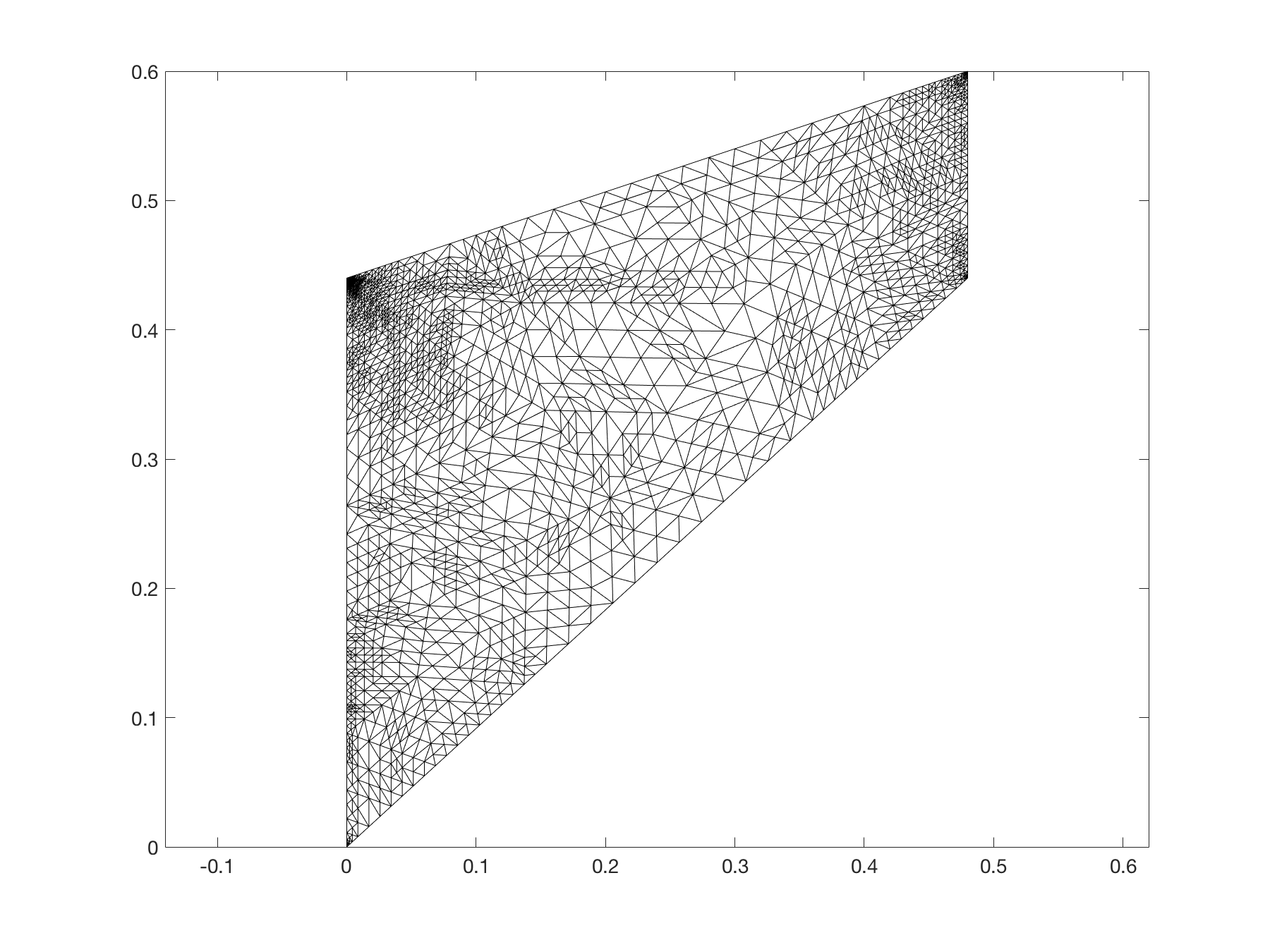}
  \caption{Triangulation after 17 adaptive refinements: $\nu = 0.5$}
  \label{fig-fe_mesh_cook5}
\end{figure}

For the incompressible case, the final triangulation after 17 adaptive refinement steps is shown in Figure
\ref{fig-fe_mesh_cook5}.
As expected, most of the refinement is happening in the vicinity of the strongest singularity at the upper left corner.

{\bf Final Remarks.} We close our contribution with remarks on the generalization to three-dimensional elasticity computations.
As already pointed out in the introduction the properties of the quadratic nonconforming element space were studied in
\cite{For:85} and its combination with piecewise linears again constitutes an inf-sup stable pair for incompressible linear elasticity.
The stress reconstruction algorithm of \cite{AinAllBarRan:12} and \cite{Kim:12a} (see the end of Section 3) can also be
generalized in a straightforward way to the three-dimensional case due to the fact that the corresponding conservation properties
hold in a similar way on elements and faces as proven in \cite{For:85}. For the improved and guaranteed estimator, the
construction of stresses with element-wise symmetry on average becomes somewhat more complicated in the three-dimensional
situation. This is due to the fact that the correction needs to be computed in the space of curls of N\'ed\'elec elements leading to
a more involved local saddle point structure. For the computation of a divergence-constrained conforming approximation we see,
however, no principal complications in three dimensions. Exploring the details of the associated analysis is the topic of ongoing
work. The presentation of the results including three-dimensional computations are planned for a future paper.

{\bf Acknowledgement.} We thank Martin Vohral{\'{\i}}k for elucidating discussions and for pointing out reference
\cite{AchBerCoq:03} to us. {\colb We are also grateful to two anonymous referees for the careful reading of our manuscript and
for helpful suggestions. In particular, both of them found an error in an earlier version of the correction procedure in Section 5.}

\newpage

\bibliography{../../biblio/articles,../../biblio/books}
\bibliographystyle{siamplain}

\end{document}